\documentclass[a4paper,11pt]{scrartcl}
\usepackage[utf8]{inputenc}
\usepackage[english]{babel}
\usepackage{amsmath, amsthm, amssymb}
\usepackage{nicefrac}
\usepackage{lipsum}
\usepackage{hyperref}
\usepackage{geometry}
\usepackage{xcolor}

\newtheorem{satz}{Theorem}[section]
\newtheorem{lemma}[satz]{Lemma}

\theoremstyle{definition}

\newtheorem{remark}[satz]{Remark}

\newtheorem{as}{Assumption}

\newcommand{\R}{\mathbb{R}}
\newcommand{\N}{\mathbb{N}}
\newcommand{\Lin}{\mathcal{L}}
\newcommand{\V}{\mathcal{V}}
\newcommand{\Hi}{\mathcal{H}}

\newcommand{\Hs}{\mathcal{HS}}
\newcommand{\Pc}{\mathcal{P}}
\newcommand{\Qc}{\mathcal{Q}}
\newcommand{\Ac}{\mathfrak{A}}
\newcommand{\Rc}{\mathcal{R}}
\newcommand{\Gc}{\mathcal{G}}
\newcommand{\Sc}{\mathcal{S}}
\newcommand{\Uc}{\mathcal{U}}
\newcommand{\Bc}{\mathcal{B}}
\newcommand{\A}{\mathcal{A}}

\newcommand{\Pctil}{\widehat{\mathcal{P}}}
\newcommand{\dual}[3][]{\! \left \langle#2,#3 \right \rangle_{#1} \!}
\newcommand{\dualb}[3][]{\big\langle #2,#3 \big\rangle_{#1}}
\newcommand{\seq}[1]{(#1_n)_{n\in \N}}
\newcommand{\seqd}[2]{(#1_{#2})_{#2 \in \N}}
\newcommand{\ska}[3][]{\left ( #2 , #3 \right )_{#1}}
\newcommand{\W}{\mathcal{W}_1(0,T)}

\newcommand{\incl}{\hookrightarrow}
\newcommand{\weak}{\rightharpoonup}
\newcommand{\weaks}{\stackrel{\ast}{\rightharpoonup}}
\newcommand{\with}{: \ }

\DeclareMathOperator{\tr}{tr}

\DeclareMathOperator{\dom}{dom}

\DeclareMathOperator*{\esssup}{ess\,sup}

\newcommand{\red}[1]{}

\numberwithin{equation}{section}

\title{\Large Convergence of the backward Euler scheme for the operator-valued
Riccati differential equation\\ with semi-definite data\footnote{The authors gratefully acknowledge financial support from the Deutsche Forschungsgemeinschaft through the
Collaborative Research Center 901 "Control of self-organizing nonlinear systems: Theoretical methods and
concepts of application" (projects A2, A8).}}

\author{\large
	Monika Eisenmann $\bullet$ Etienne Emmrich $\bullet$ Volker Mehrmann\\
\normalsize
Technische Universit\"{a}t Berlin, Institut f\"{u}r Mathematik\\[-0.5ex]
\normalsize
Stra\ss e des 17.\ Juni 136, 10623 Berlin, Germany\\[-0.5ex]
\normalsize
	\{meisenma, emmrich, mehrmann\}@math.tu-berlin.de
}

\date{\normalsize\today}
\allowdisplaybreaks

\begin{document}

	\maketitle

\begin{abstract}
  For initial value problems associated with operator-valued Riccati
  differential equations posed in the space of Hilbert--Schmidt operators
  existence of solutions is studied.
  An existence result known for algebraic Riccati equations is generalized
  and used to obtain the existence of a solution to the approximation of the
  problem via a  backward Euler scheme.
  Weak and strong convergence of the sequence of approximate solutions is established
  permitting a large class of right-hand sides and initial data.
\end{abstract}

\section{Motivation}

In this paper, we prove the existence of a solution to the initial value
problem for the operator valued Riccati differential equation
\begin{equation} \label{eq0:InitialRic}
  \begin{split}
    \Pc'(t) +  \A^*(t)\Pc(t)+ \Pc(t)\A(t) + \Pc^2(t) &= \Qc(t),
    \quad t\in
    (0,T),\\
    \Pc(0)&= \Pc_0,
  \end{split}
\end{equation}
by showing the convergence of a suitable approximation scheme. \red{Here we are particularly interested in the maximal solution.}
A solution of this problem is of
importance, for example, in the optimal control of partial differential
equations, see \cite{J.L.Lions.1971}. Even though the existence of a
solution can be deduced from the optimality conditions of a suitable control
problem, it is helpful for applications to construct a solution to the initial
value problem directly. This solution can then be used to calculate an
optimal solution of the control problem.

In the literature, different ways of examining the solvability
of the initial value problem~\eqref{eq0:InitialRic} have been studied. In
\cite{R.F.CurtainundA.J.Pritchard.1978}, an
approach using a two-parameter semigroup is proposed. This ansatz is
considered and generalized in many works, see,
e.g.,\,\cite{G.DaPratoundA.Ichikawa.1986,
G.DaPratoI.LasieckaundR.Triggiani.1985,
  F.Flandoli.1984, F.Flandoli.1993,
	I.LasieckaundR.Triggiani.1983}. A similar approach was used in
\cite{G.DaPrato.1973} based on Green functions. In \cite{L.Tartar.1974},
existence was shown by adding a holomorphic function to the linear part of
the differential equation.
In \cite{I.LasieckaundR.Triggiani.2000,I.LasieckaundR.Triggiani.2000b} many
further aspects to the approach via mild solution theory to this problem may
be found.

For matrix-valued Riccati equations, the backward Euler scheme and BDF
methods have been studied e.g. in
\cite{AscMR95,P.BennerandH.Mena.2004,
P.BennerandH.Mena.2012}.
Many further aspects related to matrix-valued Riccati equations can be
found in \cite{W.Reid.1972}. In \cite{P.BennerandH.Mena.2012}, a spatial
discretization is used to obtain results for the operator case, as well.
In \cite{E.HansenundT.Stillfjord.2014}, the convergence analysis of an
operator splitting method is considered.

\red{All these works consider solely pointwise positive definite right-hand
sides $\Qc$ and positive definite initial values $\Pc_{0}$. In the following, we
{\color{red} generalize these to a certain class of problems with
 semidefinite data.
To achieve these results we have to give up
some generality in other aspects.} In particular, results that are based on a
semigroup approach {\color{red} and deal with an underlying control
problem often  allow}
for data which do not have to be within the space of Hilbert--Schmidt
operators and therefore do not have to be compact. So far we
cannot include more general nonlinearities of the form $\Pc(t) B B^*
\Pc(t)$, $t\in [0,T]$, which arise in control problems for a suitable linear
operator $B$.  In our approach, we need to make use of the fact that
$(\Pc^2(t),\Pc(t))$ has a suitable
bound from below, where $(\cdot,\cdot)$ is the inner product within the
space of Hilbert--Schmidt operators. When including a nontrivial $B$, this
bound is no longer obtained.
}

Similarly to our approach, in \cite{R.Temam.1969, R.Temam.1971} a
numerical method to construct a
solution of~\eqref{eq0:InitialRic} is
studied. Here, the existence of a weak solution is proven via a time
discretization using a three-step splitting method. In \cite[Chapter
III.2.3 Example 3]{V.Barbu.1976}, the same result is proven through the solvability of the algebraic Riccati equation
\begin{equation*}
\tilde{A}^*P+ P \tilde{A} + P^2 = Q
\end{equation*}
for suitable linear operators $\tilde{A}$ and $Q$ and the
existence of a mild solution to the initial value problem for differential
equations with accretive operators.

In this paper  we achieve both, an existence result and a numerical
approximation method, via a discretization in time of the
initial value problem \eqref{eq0:InitialRic} using the backward Euler scheme
with constant step size $\tau = \frac{T}{N}$, $N\in \N$. This leads to the
discretized system
\begin{equation*}
\frac{P_n-P_{n-1}}{\tau} + A^*_n P_{n} + P_{n} A_n + P_{n}^2
= Q_n,\quad n=1, 2, \dots, N,
\end{equation*}
with $P_0 = \Pc_0$ and suitable operators $A^*_n$, $A_n$, and
$Q_n$.
Here $P_n$ denotes an approximation of $\Pc(t_n)$ with $t_n = n \tau$ ($n
= 1,2,\dots,N$).
Rewriting this algebraic operator equation leads to a system of algebraic
Riccati equations of the form
\begin{equation*}
\left(A_n + \frac{1}{2\tau}I  \right)^* P_{n}
+ P_{n} \left(A_n + \frac{1}{2\tau}I \right) + P_{n}^2
= Q_n + \frac{1}{\tau} P_{n-1},\quad n=1, 2, \dots, N,
\end{equation*}
with $P_0 = \Pc_0$, where in contrast to the approach in
\cite{V.Barbu.1976} the right-hand side
is more complicated. To deal with this difficulty is a major component of
our work which is assembled as follows.

In Section~\ref{section:Notation}, we
begin with a short introduction to the concept of Hilbert--Schmidt
operators
and the function spaces required for the weak solution approach.
In the next section, a generalization of the existence result for an
algebraic Riccati equation from \cite[Chapter II.3, Theorem~3.9]{V.Barbu.1976}
is considered. Here, we take the right-hand side from a class of operators
that are bounded from below by a constant depending on the operator $\A
= \A(t)$.
The class of possible right-hand sides includes operators which are not
necessarily positive definite. In Section~\ref{section:varSol}, our main result
Theorem~\ref{satz4:variationelLoesung} is presented. We prove the
existence of a solution in the weak sense, i.e., the existence of a locally integrable function
$\Pc =  \Pc(t)$ taking values in a suitable space of linear operators and fulfilling the initial condition $\Pc(0) = P_0$ in an appropriate sense such that
\begin{align*}
- \int_{0}^{T} \dual{\Pc(t)}{R} \varphi'(t) dt +
\int_{0}^{T} \dual{ \A^*(t)\Pc(t)+ \Pc(t)\A(t) +
	\Pc^2(t) }{R}
\varphi(t) dt
\\
 = \int_{0}^{T} \dual{\Qc(t)}{R} \varphi(t) dt
\end{align*}
is fulfilled for every $R$ from a suitable space of linear operators and
every smooth test function $\varphi : (0,T) \to \R$ with compact
support.
Here, $\dual{\cdot}{\cdot}$ is a suitable
duality pairing that we will introduce in Section
\ref{section:Notation} in more detail.
To this extent, we construct a solution of the initial value problem for the
Riccati differential equation using the backward Euler scheme. The
time-discretized equations that occur can be solved using the new existence
result for algebraic Riccati equations from Section~\ref{section:algRic} which
allows us to consider a larger class of functions for the
right-hand side and the initial value in comparison with
\cite{R.Temam.1971}. We can even allow indefiniteness for the data as well
as a more general condition on the
integrability of the right-hand side which  are the same as proposed in
\cite[Section~19--20]{L.Tartar.2006}.

Even though a variational approach is restricted to data that are
Hilbert--Schmidt operators, which are compact operators, we believe that this
concept of solution is more suitable for numerical examinations. For
example, it offers both a constructive scheme and the possibility to fully
discretize the problem.

\section{Notation and preliminaries} \label{section:Notation}

In order to state the weak formulation of the problem, a Hilbert space
setting within the space of linear operators is required. To proceed like this,
we briefly introduce the space of Hilbert--Schmidt operators. A complete
introduction can be found in \cite{N.DunfordundJ.Schwartz.1963}.

For real, separable Hilbert spaces $(X, \ska[X]{\cdot}{\cdot}, \|\cdot\|_X)$
and $(Y, \ska[Y]{\cdot}{\cdot}, \|\cdot\|_Y)$, let $\Lin(X,Y)$ be the space of
linear bounded operators mapping $X$ into $Y$. We denote the
\emph{Hilbert--Schmidt
norm} of a compact operator $T \in \Lin(X,Y )$ by
\begin{align*}
\|T\|_{\Hs(X,Y)} = \left( \sum_{n=1}^{\infty}
\|Te_n\|^2_{Y}\right)^{\frac{1}{2}},
\end{align*}
where $\seq{e}$ is an arbitrary orthonormal basis of $X$. This norm is
induced by the inner product
\begin{align*}
\ska[\Hs(X,Y)]{T}{S}=\sum_{n=1}^{\infty} \ska[Y]{Te_n}{Se_n} = \tr(T^*S) =
\tr(ST^*), \quad S, T \in \Lin(X,Y ),
\end{align*}
where the operator $T^*$ denotes the Hilbert space adjoint of $T$ and
$\tr$ is the trace of an operator.
Equipped with this inner product and norm, the linear space
\begin{align*}
\Hs(X,Y) = \left\{ T\in \Lin(X,Y) : \| T\|_{\Hs(X,Y)} < \infty \right\}
\end{align*}
of Hilbert--Schmidt operators is a separable Hilbert space. Note that both
the norm $\|\cdot\|_{\Hs(X,Y)}$ and the inner product
$\ska[\Hs(X,Y)]{\cdot}{\cdot}$ are independent of the choice of basis
$(e_n)_{n\in \N}$ as shown in \cite[Theorem~3.6.1]{B.Simon.2015}.

Using this closed subspace of the space of linear operators, we can
introduce a suitable \emph{Gelfand triple} for the following theory.
This concept to construct suitable spaces can also be found in
\cite{I.G.Rosen.1991} and \cite{R.Temam.1971}.

We begin by determining the exact assumptions on the spaces on which
the operators are defined.

\begin{as} \label{as:spaces}
	Let $(V, \ska[V]{\cdot}{\cdot}, \|\cdot\|_V)$ and $(H,\ska[H]{\cdot}{\cdot},
	\|\cdot\|_H)$ be real, separable Hilbert spaces such that $V$ is compactly
	and densely embedded (denoted as $\stackrel{c,d}{\incl}$ ) into $H$.

	The embedding constant is denoted by 	$C_{V,H}$, i.e.,
	\begin{equation} \label{eq2:normHV}
	\|v\|_{H} \leq C_{V,H} \|v \|_{V}
	\end{equation}
	holds for every $v \in V$.
\end{as}

Identifying $(H, \ska[H]{\cdot}{\cdot}, \|\cdot\|_H)$ with its dual $(H^*,
\ska[H^*]{\cdot}{\cdot}, \|\cdot\|_{H^*})$, we have the Gelfand triple
\begin{equation} \label{eq2:DefVH}
V \stackrel{c,d}{\incl} H \cong H^*\stackrel{c,d}{\incl} V^*,
\end{equation}
where $(V^*, \ska[V^*]{\cdot}{\cdot}, \|\cdot \|_{V^*})$ denotes the dual
space of $V$.
The inner product in $H$ can be extended to the duality pairing between
$V^*$ and $V$, which is denoted by $\dual[V^*\times V]{\cdot }{\cdot}$.
We introduce the Hilbert spaces $\V$, $\Hi$ and $\V^*$ as
\begin{equation} \label{eq2:DefVHi}
\V = \Hs(V^*, H) \cap \Hs(H, V),  \quad \Hi  = \Hs(H,H) \ \text{ and }
\
\V^* = \Hs(V, H) + \Hs(H, V^*)
\end{equation}
and identify $\Hi$ with its dual $\Hi^*$.
Note that it is possible to define the above intersection of vector
spaces, since every appearing space can be interpreted as a subset of
$\Lin(V,V^*)$. For further details, see,
e.g.~\red{\cite[Chapter 2.3]{J.BerghandJ.Lofstrom.1976} or
}\cite[Bemerkung 5.11]{GGZ.1974}.
\red{The spaces $\V$ and $\Hi$ are equipped with the inner products
\begin{align*}
\ska[\V]{R}{S}
= \ska[\Hs(V^*, H)]{R}{S} + \ska[\Hs(H, V)]{R}{S}, \quad
\ska{R}{S}
= \ska[\Hs(H,H)]{R}{S},
\end{align*}
which also induce the norms $\|\cdot\|_{\V}$ and $\|\cdot\|_{\Hi}$,
respectively.
Furthermore, $\|\cdot\|_{\V^*}$ is the norm induced on the dual space
$\V^*$ of $\V$.}
It is important to note that the norm $\|\cdot\|_{\Hi}$ is submultiplicative (see, e.g., \cite[Theorem~4 in Ch.~XI.6.3]{N.DunfordundJ.Schwartz.1963}).
The inner product in $\Hi$ can be extended to the duality pairing between
$\V^*$ and $\V$ that is denoted by
\begin{align*}
\dual{\cdot }{\cdot}:  \V^* \times \V \to \R.
\end{align*}
As we have required the embedding $V \stackrel{c}{\incl} H$ to be compact,
this property extends to the operator spaces $\V$ and $\Hi$, i.e., the
embeddings
\begin{align*}
\V \stackrel{c}{\incl} \Hi \cong \Hi^* \stackrel{c}{\incl} \V^*
\end{align*}
are compact as well (see \cite[Proposition 2.1.]{R.Temam.1971}). Also the
norm estimate \eqref{eq2:normHV} transfers to $\V$ and $\Hi$:
 $ \|S\|_{\Hi} \leq \frac{C_{V,H}}{\sqrt{2}}
\|S \|_{\V}$ for all $S\in \V$.

For $p \in[1,\infty]$ and a Banach space $(X,\|\cdot\|_X)$, we introduce the
Bochner-Lebesgue space
\begin{align*}
L^p(0,T;X) = \left\{ u:[0,T] \to X : u \text{ is strongly measurable  and
}\|u\|_{L^p(0,T;X)} < \infty \right\}
\end{align*}
with the usual norm given by
\begin{align*}
\|u\|_{L^p(0,T;X)} =
\begin{cases}
\begin{displaystyle}
\left( \int_{0}^{T} \|u(t)\|_{X}^p dt \right)^{\frac{1}{p}}
\end{displaystyle}
\quad &\text{ for } p\in [1,\infty),\\
\begin{displaystyle}
\esssup_{t\in [0,T]} \|u(t)\|_{X}
\end{displaystyle}
\quad &\text{ for } p = \infty.
\end{cases}
\end{align*}
Further, we introduce the space $L^1(0,T;\Hi) + L^2(0,T;\V^*)$  which consists of the functions
$u: [0,T] \to \V^*$ such that there exist $u_1 \in L^1(0,T;\Hi)$ and $u_2 \in L^2(0,T; \V^*)$
that fulfill $u = u_1 + u_2$  . A norm for this space is given by
\begin{align*}
  \| u \|_{L^1(0,T;\Hi) + L^2(0,T;\V^*)}
	= \inf_{\substack{u_1 \in L^1(0,T;\Hi)\\
		u_2 \in L^2(0,T;\V^*)\\ u = u_1 +u_2}}  \left( \|u_1\|_{L^1(0,T;\Hi)}
+ \|u_2\|_{L^2(0,T;\V^*)} \right),
\end{align*}
see  \cite[Section~20]{L.Tartar.2006} for more details. This space in mind,
we define
\begin{equation}\label{eq2:DefW1}
\W = \left\{ u \in L^2(0,T;\V):u' \in L^1(0,T;\Hi) + L^2(0,T;\V^*) \right\}
\end{equation}
equipped with the norm
\begin{align*}
\|u\|_{\W} = \|u\|_{L^2(0,T;\V)} + \| u' \|_{L^1(0,T;\Hi) + L^2(0,T;\V^*)}.
\end{align*}
Note that $C^{\infty}([0,T]; \V)$ is dense in $\W$. Here,
$C^{\infty}([0,T]; \V)$ is the linear space of all infinitely many times
differentiable functions mapping $[0,T]$ into $\V$.
Furthermore, $\W$ is dense and continuously embedded into the space
$C([0,T];\Hi)$ of continuous functions mapping $[0,T]$ into $\Hi$. See
\cite[Section~20]{L.Tartar.2006} for more details.
By $C_c^{\infty}(0,T)$, we denote the linear space of infinitely many times
differentiable functions mapping $(0,T)$ into $\R$ having compact support.

In the following, it will be of importance to distinguish between different
types of convergence of a sequence $\seq{x}$ to a limit $x$ in a Banach
space $(X, \|\cdot\|_X)$. To this end, we will briefly explain the used
notation. We denote the (strong) convergence by
\begin{align*}
x_n \to x \text{ in } \ X, \quad \text{ i.e., } \quad \|x - x_n \|_X \to 0
\end{align*}
as $n \to \infty$, the weak convergence by
\begin{align*}
x_n \weak x \text{ in } \ X, \quad \text{ i.e., } \quad  \dual[X^*\times
X]{f}{x - x_n} \to 0
\ \text{ for every } \ f \in X^*
\end{align*}
as $n \to \infty$ and the weak$\ast$ convergence of a sequence $\seq{f}$
to a limit $f$ in $X^*$ by
\begin{align*}
f_n \weaks f \text{ in } \ X^*,
\quad \text{ i.e., } \quad
\dual[X^*\times X]{f - f_n}{x} \to 0 \ \text{ for every } \ x \in X.
\end{align*}
Here $X^*$ denotes the dual space of $X$ and $\dual[X^*\times
X]{\cdot}{\cdot}$ the duality pairing.
For more details, see, e.g., \cite[Chapter 3]{H.Brezis.2010}.
In the following, $c>0$ always denotes a generic constant.

\section{An algebraic operator Riccati equation} \label{section:algRic}

In this section we prove a generalization of the existence result for
algebraic Riccati equations given in \cite[Chapter II.3, Lemma 3.2 and
Theorem~3.9]{V.Barbu.1976}.

\begin{as}\label{as:AalgRic}
	Let $A \in  \Lin(V,V^*)$ be a \emph{strongly positive operator}, i.e., there exists
	a constant $\mu_{V}>0$ such that
	\begin{equation} \label{eq3:AstronglyPos}
	\dual[V^*\times V]{Au}{u} \geq \mu_V \|u\|_{V}^2
	\end{equation}
	holds for all $u \in V$.
\end{as}

\begin{remark} \label{remark:muH}
	Note that Assumptions~\ref{as:spaces}.\:and \ref{as:AalgRic}.\:immediately
	imply that
	\begin{align} \label{eq3:AstronglyPosH}
	\dual[V^*\times V]{Au}{u} \geq \mu_H \|u\|_{H}^2
	\end{align}
	for all $u \in V$ with $\mu_H = \frac{\mu_V}{C_{V.H}^2}$. Nevertheless
	\eqref{eq3:AstronglyPosH} may also be fulfilled with a constant $\mu_H$
	much larger than $\frac{\mu_V}{C_{V.H}^2}$. This, indeed, will
	be the case in our application in Lemma~\ref{lemma4:Apriori} below.
\end{remark}

\begin{remark}
  \red{Note that this positivity condition is not the same as the  positivity for
  a self-adjoint operator that we will introduce further below in
  \eqref{eq3:PositivityP}. For a self-adjoint operator, positivity is equivalent
  to solely nonnegative spectrum. If an operator is not self-adjoint,
  a positive spectrum is not enough to induce strong positivity. Therefore,
  Assumption~\ref{as:AalgRic} is stronger. A further discussion on this topic
  can be found, for example, in \cite[Chapter VI, Remark 2.29]{T.Kato.1995}.}
\end{remark}

\begin{lemma} \label{lemma3:AstarkPos}
	Let Assumptions~\ref{as:spaces}.\:and \ref{as:AalgRic}.\:be satisfied. Then
	the	estimate
	\begin{align*}
	\dual{P A}{P} \geq \mu_V \|P\|^2_{\Hs(V^*,H)}
	\end{align*}
  holds for all \red{$P \in \V$}. Moreover,
  \begin{align*}
	\dual{A^* P}{P} \geq \mu_V \|P\|^2_{\Hs(H, V)}
	\end{align*}
  holds for all \red{$P \in \V$}, where $A^*$ is the dual operator of $A$.
	In particular,
  \begin{align*}
	\dual{A^* P_1 + P_1A}{P_1} \geq \mu_V \|P_1\|^2_{\V}
	\quad \text{ and } \quad
	\ska{A^* P_2 + P_2A}{P_2} \geq 2\mu_H\|P_2\|^2_{\Hi}
	\end{align*}
  hold for all $P_1 \in \V$ and $P_2 \in \{ R \in \V:  A^* R + RA \in \Hi
	\}$.
\end{lemma}
\begin{proof}
	A proof can be found in \cite[Chapter II.3, Lemma 3.4]{V.Barbu.1976}.
\end{proof}

The following results will be stated for (indefinite) self-adjoint operators in
$\Hi$ that
have a lower bound for the  possibly negative eigenvalues.
Let $P \in \Hi$ be self-adjoint. Then it is also compact, since $P$ is a
Hilbert--Schmidt operator. Applying the Hilbert--Schmidt theorem (see
\cite[Theorem~5, Chapter VII.4.5]{N.DunfordundJ.Schwartz.1957}) assures
the existence of an orthonormal system $\left(e_n\right)_{n\in\N}$ of
eigenvectors of $P$ in $H$. Using the eigenvectors
$\left(e_n\right)_{n\in\N}$, $P$ can be represented through
\begin{align} \label{eq3:repP}
P = \sum_{n=1}^{\infty} \alpha_n \ska[H]{e_n}{\cdot} e_n,
\end{align}
where each $\alpha_n \in \R$, $n\in \N$, is an eigenvalue of the operator $P$.
As every self-adjoint operator in $\Hi$ has a representation like this, we can
introduce a lower bound for self-adjoint operators.
For $\gamma \in \R$, we write $P \geq \gamma$ if every
eigenvalue of $P$ is greater or equal to $\gamma$. This is equivalent to the
condition $P - \gamma I\geq 0$, i.e.,
\begin{align}\label{eq3:PositivityP}
\ska[H]{(P - \gamma I)u}{u} \geq 0
\end{align}
for all $u \in H$.
This in mind, we introduce the set
\begin{equation*}
C_{-\gamma} = \{ P \in \Hi \with P = P^*, P \geq -\gamma \}
\end{equation*}
of operators whose eigenvalues are bounded from below by $-\gamma$.

For the linear part of the Riccati equation we introduce the linear \emph{Lyapunov} operator
\begin{equation*}
\Ac_0  : \V \to \V^* , \quad  P \mapsto A^* P + P A
\end{equation*}
but use the same notation also for its restriction
\begin{equation*}
\Ac_0: \dom(\Ac_0) \subseteq \Hi \to \Hi, \quad P \mapsto A^* P + P A
\end{equation*}
with $\dom(\Ac_0) = \{ P \in \V \with \Ac_0 P \in \Hi  \}$.
If Assumption~\ref{as:AalgRic}.\:holds, then due to the positivity of $A$ and
$A^*$ and Lemma~\ref{lemma3:AstarkPos}, the operator $\Ac_0$ is
\emph{m-accretive} and $\ska{\Ac_0 P }{P}\geq 0$ is fulfilled for every $P \in
\dom(\Ac_0)$.
Therefore, $\lambda \Ac_0 +  I$ is surjective for every $\lambda>0$, see,
e.g., \cite[Theorem~2.2]{V.Barbu.2010}. The
following lemma provides some important properties for the resolvent of
$\lambda \Ac_0$, $\lambda>0$.

\begin{lemma} \label{lemma3:ALipPos}
	Let Assumptions~\ref{as:spaces}.\:and \ref{as:AalgRic}.\:be satisfied and $\lambda >0$. Then $(\lambda \Ac_0 + I)^{-1} :
	\Hi \to \dom(\Ac_0)$ is well-defined, linear, bounded and fulfills the
	estimate
	\begin{align*}
	\| (\lambda \Ac_0 + I)^{-1} Q \|_{\Hi} \leq \frac{1}{1 + 2 \lambda \mu_H}
	\|Q\|_{\Hi}
	\end{align*}
	for all $Q \in \Hi$.  Furthermore, for $Q \in C_{-\gamma}$
  \begin{align*}
	(\lambda \Ac_0 + I)^{-1}Q \in C_{-\frac{\gamma}{1+ 2 \lambda \mu_H} }
	\end{align*}
  is fulfilled.
\end{lemma}
\begin{proof}
	The operator
	\begin{align*}
	\lambda  \Ac_0 + I: \V \to \V^*, \quad R \mapsto \lambda (A^* R + R A) + R
	\end{align*}
  is linear, bounded, and strongly positive, \red{since}
  \begin{align*}
	\dual{(\lambda  \Ac_0 + I) R}{R }
	\geq \lambda \mu_V \|R\|_{\V}^2 + \|R\|^2_{\Hi}
	\geq \lambda \mu_V \|R\|_{\V}^2
	\end{align*}
  \red{is fulfilled for every $R \in \V$.} Therefore, $\lambda  \Ac_0 + I: \V \to
  \V^*$  is bijective due
  to the
	Lax-Milgram lemma, and for every $Q\in \Hi \subseteq \V^*$ there exists a
	unique $P \in \dom(\Ac_0)$ such that
  \begin{align*}
	\left(\lambda \Ac_0 +I\right) P = Q.
	\end{align*}
  This proves the existence of the linear operator $(\lambda  \Ac_0 +
	I)^{-1}$ in $\Hi$. For a self-adjoint operator $Q$, it follows that
  \begin{equation*}\label{eq3:LinAlgEq}
  \left(\lambda \Ac_0 +I\right) P =
	Q = Q^* = \left(\left(\lambda \Ac_0 +I\right) P \right)^*= \lambda (A^*P^* +
	P^*A) + P^* =  \left(\lambda \Ac_0 +I\right) P^* \, .
	\end{equation*}
  Since the mapping $\lambda \Ac_0 + I: \V \to\V^* $ is injective, the
	operators $P$ and $P^*$ coincide, i.e., $P$ is self-adjoint as it is a
	bounded, symmetric 	operator in $H$. In the following, we represent
	$P$ as in \eqref{eq3:repP}.
	Exploiting the fact that the elements $e_n$, $n\in \N$, from
	\eqref{eq3:repP} are eigenvectors of $P$, we obtain that
	\begin{align*}
	\ska[H	]{Qe_n}{e_n}
	&= \lambda \dual[V^*\times V]{A^*Pe_n}{e_n} +\lambda
	\dual[V^*\times V]{PAe_n}{e_n}  + \ska[H]{Pe_n}{e_n} \\
	&=2 \lambda \alpha_n \dual[V^*\times V]{Ae_n}{e_n} + \alpha_n.
	\end{align*}
  This implies that
  \begin{align*}
	\alpha_n = (1+ 2 \lambda \dual[V^*\times V]{Ae_n}{e_n})^{-1}
	\ska[H]{Qe_n}{e_n},
	\end{align*}
  which leads to the estimate
  \begin{align*}
	\| (\lambda \Ac_0 + I)^{-1} Q \|_{\Hi}^2 = \| P \|_{\Hi}^2
	&= \sum_{n=1}^{\infty} \alpha_n^2
	= \sum_{n=1}^{\infty} \left( \frac{\ska[H]{Qe_n}{e_n}}{1+ 2 \lambda
		\dual[V^*\times V]{Ae_n}{e_n}} \right)^2 \\
	&\leq \sum_{n=1}^{\infty} \left( \frac{\ska[H]{Qe_n}{e_n}}{1+ 2 \lambda
	\mu_H		\|e_n \|^2_H } 	\right)^2
	\le \left( \frac{1}{1+ 2 \lambda \mu_H} \right)^2 \|Q\|_{\Hi}^2 .
	\end{align*}
  Since we have the lower bound $\ska[H]{Qe_n}{e_n} \geq -\gamma$ for
  $Q\in C_{-\gamma}$, we obtain that
	\begin{align*}
	\alpha_n = (1+ 2 \lambda \dual[V^*\times V]{Ae_n}{e_n})^{-1}
	\ska[H]{Qe_n}{e_n} \geq -\frac{\gamma}{1+ 2 \lambda \mu_H}
	\end{align*}
  for every $n\in \N$. As $P = \left(\lambda \Ac_0 +I\right)^{-1} Q$ is
  self-adjoint by definition and has
	a suitable lower bound for every eigenvalue, this proves the second
	assertion of the lemma.
\end{proof}

In the following, we derive some corresponding results for the full
Riccati equation.
In order to do this, we define the nonlinear operator
\begin{equation}\label{eq3:defB}
\Bc : \Hi \to \Hi, \quad P \mapsto P^2.
\end{equation}

\begin{lemma} \label{lemma3:QuadTermNeg}
	Let Assumption~\ref{as:spaces}.\:be satisfied and let $\gamma \geq
	0$ be	given.
	If $P\in C_{-\gamma}$ then
	\begin{align*}
	\ska{PR}{R} \geq -\gamma \|R\|_{\Hi}^2 \quad \text{ and } \quad \ska{RP}{R}
	\geq -\gamma \|R\|_{\Hi}^2
	\end{align*}
  is fulfilled for every $R \in \Hi$ and
  \begin{align*}
	\dual{P R + RP}{R} \geq - \gamma C_{V,H}^2 \|R\|^2_{\V}
	\end{align*}
  is fulfilled for every $R \in \V$.
\end{lemma}
\begin{proof}
	Since we assume that $P \in C_{-\gamma}$, the operator $P + \gamma I$ is
	positive such that
	\begin{align*}
	\ska{(P + \gamma I) R}{R}
	\geq 0
	\end{align*}
	for all $R \in \Hi$. This implies, in particular, the first assertion.
	The other assertions follow in an analogous manner.
\end{proof}

For $\lambda >0$, the operator $ \lambda \Bc + I$ is neither linear nor
Lipschitz continuous on $\Hi$. Still the following lemmas provide a suitable
boundedness and Lipschitz continuity of the inverse of the restriction of
the operator onto a suitable domain $C_{-\gamma}$, $\gamma>0$.
\red{Note that it is of particular importance that $C_{-\gamma}$ consists of
  self-adjoint Hilbert--Schmidt operators. For each operator, we therefore
  find a sequence of square summable eigenvalues, which is a
  crucial argument for some technical details in the following proofs.}

\begin{remark} \label{remark:1dEx}
	In order to get a better intuition for such an inverse of a restriction of the operator $ (\lambda \Bc + I)$ for
	given $\lambda >0$, it is
	helpful to consider the problem in $\R$ at first. We want to find $P \in \R$
	such that for given $\Qc \in \R$
	\begin{align*}
	( \lambda \Bc + I)P = \lambda P^2  + P= Q.
	\end{align*}
	This equation can easily be solved in $\R$ if $\lambda >0$ is
	small enough and has the solutions
	\begin{align*}
	P = \frac{1}{2\lambda } \left(-1\pm \sqrt{1 + 4\lambda Q }
	\right).
	\end{align*}
	For $Q > 0$, this admits one positive and one negative solution.
	For $\lambda Q \in \left(-\frac{1}{4},0\right)$, the problem has two
	different negative solutions.
\end{remark}

Analogously to Remark~\ref{remark:1dEx}, looking at the operator setting
again, the equation
 \begin{align} \label{eq3:equationB}
 (\lambda \Bc + I) P =  \lambda P^2 + P = Q.
  \end{align}
possesses only one positive solution $P$ if $Q \geq 0$ but more than one
negative solution if $Q <	0$.	In the following, we will always work with
the \red{maximal} solution.

\begin{lemma} \label{lemma3:BbdNegCoefB}
  {Let Assumption~\ref{as:spaces}.\:be satisfied and let
  $\gamma \geq 0$ be given.
  Then for $\lambda \in(0,\frac{1}{4\gamma})$ if $\gamma >0$, and
  $\lambda \in(0,\infty)$ if $\gamma = 0$, and every $Q\in C_{-\gamma}$
  there exists a unique $P \in C_{\frac{ -2\gamma }{1 + \sqrt{1 - 4 \lambda
  \gamma } } }$ such that \eqref{eq3:equationB}
  is fulfilled. In the following, $P$ is denoted by $(\lambda \Bc + I)^{-1} Q$.
  Moreover, for $Q,Q_1,Q_2 \in C_{-\gamma}$ the estimates
  \begin{align}\label{eq4:Bbound}
  \| (\lambda \Bc + I)^{-1} Q\|_{\Hi}
  \leq \frac{ 2}{ 1 + \sqrt{1 - 4\lambda \gamma }}
  \|Q\|_{\Hi}
  \end{align}
  and
  \begin{align}\label{eq4:BLip}
  \| (\lambda \Bc + I)^{-1} Q_1 - (\lambda \Bc + I)^{-1} Q_2 \|_{\Hi} \leq
  \frac{1 + \sqrt{ 1 - 4\lambda \gamma } }{ 1 + \sqrt{1 -
      4\lambda \gamma }
    -4 \lambda	\gamma} \| Q_1 - Q_2 \|_{\Hi}
  \end{align}
  hold true.}
\end{lemma}
\begin{proof}
  For $Q \in C_{-\gamma}$, we begin by constructing a solution $P$ for
  \eqref{eq3:equationB}. We demonstrate the uniqueness of such an element $P$ at the end of the proof.
  Since $Q \in \Hi$ is self-adjoint and \red{a Hilbert--Schmidt operator, it is
  in particular compact. Thus} there exists an orthonormal system $\seq{e}$
  of eigenvectors and a sequence $\seq{\nu}$ of real eigenvalues such that
  \begin{align*}
  Q  = \sum_{n=1}^{\infty} \nu_n \ska[H]{e_n}{\cdot} e_n \, ,
  \end{align*}
where $\nu_n \ge - \gamma$.
  Then the self-adjoint operator $P$ given by
  \begin{equation}\label{eq3:defP}
  P = \sum_{n=1}^{\infty} \alpha_n \ska[H]{e_n}{\cdot} e_n
  \quad \text{with} \quad
  \alpha_n = \frac{1}{2\lambda} \left(-1 + \sqrt{1 + 4 \lambda  \nu_n}
  \right)
  \end{equation}
  fulfills \eqref{eq3:equationB}, as can be seen by simply inserting $P$
  into \eqref{eq3:equationB}. \red{We will argue that this $P$ is the unique
  solution at the end of the proof.}
  Since $\inf_{n\in \N} \nu_n \geq -\gamma $ and $\lambda
  < \frac{1}{4 \gamma }$ imply that $1 + 4\lambda \nu_n > 0$
  for every $n\in \N$, the numbers $\alpha_n$, $n\in \N$, are indeed real.
  We obtain that the eigenvalues of $P$ satisfy
  \begin{align*}
  \alpha_n
  &= \frac{1}{2\lambda} \left(- 1 + \sqrt{1 + 4\lambda
    \nu_n } \right)\\
  &= \frac{ \left(-1 + \sqrt{1 + 4\lambda
      \nu_n } \right) \left(1 + \sqrt{1+ 4\lambda
      \nu_n} \right)}
  {2\lambda  \left(1 + \sqrt{1 + 4\lambda \nu_n } \right)}\\
  &	= \frac{ 2\nu_n }{1 + \sqrt{1 + 4\lambda \nu_n
  } }\\
  &	\geq \frac{ - 2 \gamma }{1 + \sqrt{1 + 4\lambda \nu_n } }
  \geq \frac{ -2\gamma   }{1 + \sqrt{1 - 4 	\lambda \gamma  } } \, .
  \end{align*}
Thus, $P \in  C_{\frac{-2\gamma   }{1 + \sqrt{1 - 4 \lambda \gamma} } }$ is fulfilled if
  \begin{align}\label{eq4:proofBbound}
  \| P\|_{\Hi}^2
  = \sum_{n=1}^{\infty} \alpha_n^2
  = \sum_{n=1}^{\infty} \left(\frac{ 2\nu_n }{1 + \sqrt{1 +
      4\lambda \nu_n }} \right)^2
    \leq \left(\frac{ 2  }{1 + \sqrt{1 -
      4\lambda \gamma }} 	\right)^2 \|Q\|_{\Hi}^2 < \infty, \,
  \end{align}
  \red{i.e., the Hilbert--Schmidt norm is finite.}
  In the following, we consider $Q_1,Q_2 \in
  C_{-\gamma}$ and choose $P_1,P_2 \in  C_{\frac{ -2\gamma   }{1 +
  \sqrt{1 - 4 \lambda \gamma} } }$ such that
  \begin{align*}
  (\lambda \Bc + I) P_1 =  Q_1
  \quad \text{and} \quad
  (\lambda \Bc + I) P_2 =  Q_2
  \end{align*}
  and obtain that
  \begin{align}
  \notag
  \| Q_1 - Q_2 \|_{\Hi} \|P_1 - P_2\|_{\Hi}
  &= \|(\lambda \Bc + I) P_1 - (\lambda \Bc + I) P_2 \|_{\Hi} \|P_1 -
  P_2\|_{\Hi} \\
  \notag
  &\geq \ska{(\lambda \Bc + I) P_1 - (\lambda \Bc + I) P_2}{P_1 - P_2}\\
  &= \lambda \ska{P_1^2 - P_2^2}{P_1 - P_2} + \ska{P_1 -P_2}{P_1 -	P_2}.
  \label{eq3:lemSum1}
  \end{align}
  Since $P_1,P_2 \geq \frac{-2\gamma }{1 + \sqrt{1 - 4\lambda \gamma } }$,
  the first summand on the right-hand side of \eqref{eq3:lemSum1} can
  be estimated by
  \begin{align*}
  \ska{P_1^2 - P_2^2}{P_1 - P_2}
  &= \ska{P_1 (P_1 - P_2)}{P_1 - P_2} +\ska{(P_1 - P_2) P_2 }{P_1 - P_2} \\
  & \geq \frac{-4\gamma }{1 + \sqrt{1 - 4\lambda
      \gamma} } \|P_1 - P_2 \|_{\Hi}^2.
  \end{align*}
  Altogether, this implies
  \begin{align*}
  \| Q_1 - Q_2 \|_{\Hi} \|P_1 - P_2\|_{\Hi}
  \geq \frac{-4\lambda \gamma }{1 + \sqrt{1 - 4\lambda \gamma} } \|P_1 - P_2
  \|_{\Hi}^2 + \|P_1 - P_2 \|_{\Hi}^2,
  \end{align*}
  as well as
  \begin{align} \label{eq4:proofBLip}
  \| Q_1 - Q_2 \|_{\Hi}
  \geq \frac{ 1 + \sqrt{1 - 4\lambda \gamma } -4\lambda
    \gamma }{1 + \sqrt{1 - 4\lambda \gamma} } \|P_1 - 	P_2\|_{\Hi}.
  \end{align}
  This inequality in mind, it now follows directly that the operator $P$
  defined in  \eqref{eq3:defP} is the unique solution to $(\lambda \Bc + I) P =
  Q$ within the set $C_{ \frac{-2\gamma }{1 + \sqrt{1 - 4\lambda \gamma }
  }}$, since
  \begin{equation*}
  \frac{ 1 + \sqrt{1 - 4\lambda \gamma } -4\lambda
    \gamma }{1 + \sqrt{1 - 4\lambda \gamma} } > 0 \, .
  \end{equation*}
The estimates \eqref{eq4:Bbound} and \eqref{eq4:BLip} are given by \eqref{eq4:proofBbound} and \eqref{eq4:proofBLip}, respectively.
\end{proof}

The previous lemmas in mind, we can now prove the main result of this
section. \red{We prove the existence of a solution to the algebraic Riccati
equation. This equation is not uniquely solvable in $\Hi$ but as mentioned in
Remark~\ref{remark:1dEx}, we consider the maximal solution within the
set $C_{-\gamma}$.}
\begin{satz} \label{satz3:AlgRicNeg}
	Let Assumptions~\ref{as:spaces}.\:and \ref{as:AalgRic}.\:be	satisfied and
	let $\gamma$ with $0 \leq \gamma < \mu_H $ be 	given.
  Then for $Q \in C_{-\gamma^2}$ there exists $P \in \V \cap C_{-\gamma}$
  such that
	\begin{equation*}
	A^*P + PA + P^2 = Q.
	\end{equation*}
\end{satz}
\begin{proof}
	The proof is a generalization of the proof in \cite[Chapter II.3, Lemma
	3.2 and Theorem~3.9]{V.Barbu.1976}.
  \red{The main idea of the proof is to approximate the Riccati equation
  using a Yosida approximation of the quadratic term and to show, via a fixed
  point argument,
  that the approximate equation is uniquely solvable in a suitable set.
  It then remains to prove
  that the limit of the sequence of approximate solutions solves the actual
  equation. Here, it is of importance to make use of the fact that the set
  $C_{-\gamma}$ consists of selfadjoint Hilbert--Schmidt operators that we
  can represent using their eigenvalues and eigenvectors. In particular, this
  offers the possibility to prove that the solution is bounded from below.}

  \red{To prove the existence of a solution to the approximate equation,
  we}
  define the operator $\Gc: C_{-\gamma} \to \Hi$ for $\lambda
  \in(0,\frac{1}{4\gamma})$ if $\gamma >0$, and $\lambda \in(0,\infty)$ if
  $\gamma = 0$, through
	\begin{align*}
	\Gc(P) = \lambda\left( \lambda \Ac_0 + I\right)^{-1} Q + \left( \lambda
	\Ac_0 + I\right)^{-1} \left(\lambda \Bc + I \right)^{-1} P.
	\end{align*}
	In the following, we use the Banach fixed-point theorem to prove the existence of a fixed-point. We begin by proving
	that $\Gc$ maps the closed subset $C_{-\gamma}$ of $\Hi$
	into itself.
For $Q \in C_{-\gamma^2}$ and $P \in C_{-\gamma}$, from Lemma~\ref{lemma3:ALipPos} and Lemma~\ref{lemma3:BbdNegCoefB} we obtain that
	\begin{align*}
	\lambda \left( \lambda \Ac_0 + I \right)^{-1} Q \geq
	\frac{ - \lambda\gamma^2 }{1 + 2\lambda \mu_H},\\
	\left( \lambda \Bc + I \right)^{-1} P
	\geq \frac{ -2\gamma }{1 + \sqrt{1 - 	4\lambda \gamma }},
	\end{align*}
	and
	\begin{align*}
	\left(\lambda \Ac_0 + I \right)^{-1} \left( \lambda \Bc + I \right)^{-1} P
	\geq \frac{1}{ 1 + 2\lambda \mu_H}
	\cdot \frac{ -2\gamma }{1 + \sqrt{1 - 4\lambda \gamma } }.
	\end{align*}
	If $\lambda > 0$ is sufficiently small, then
	\begin{align*}
	\gamma \geq \frac{\lambda\gamma^2 }{1 + 2\lambda \mu_H}  +
	\frac{2\gamma}{(1 + 2\lambda \mu_H) (1+ 	\sqrt{1- 4\lambda
			\gamma}) }.
	\end{align*}
	This yields $\Gc(P) \ge - \gamma$ and thus shows that $\Gc$ maps $C_{-\gamma}$ into itself.

The next step is to prove that $\Gc$ is a contraction on
	$C_{-\gamma}$. Using
	Lemma~\ref{lemma3:ALipPos} and Lemma~\ref{lemma3:BbdNegCoefB}, for
	$P_1,P_2 \in C_{-\gamma}$ it follows that
	\begin{align*}
	&\|\Gc(P_1) - \Gc(P_2) \|_{\Hi} \\
	&\quad= \| \left( \lambda \Ac_0 + I \right)^{-1} \left(\lambda \Bc + I
	\right)^{-1}
	P_1 - \left( \lambda \Ac_0 + I \right)^{-1} \left(\lambda \Bc + I
	\right)^{-1}
	P_2  \|_{\Hi} \\
	&\quad\leq \frac{1}{1+2\lambda\mu_H} \cdot \frac{1 + \sqrt{1 - 4\lambda
			\gamma} }{
		1 + \sqrt{1 - 4\lambda \gamma } -4\lambda	\gamma }  \| P_1 - P_2 \|_{\Hi},
	\end{align*}
	where
	\begin{align*}
	&\frac{1}{1+2\lambda\mu_H}\cdot \frac{1 + \sqrt{1 - 4\lambda \gamma } }{ 1
	+ \sqrt{1 - 4\lambda \gamma } -4 \lambda \gamma }\\
	&\quad=  \frac{ 1 + \sqrt{1 - 4\lambda \gamma } }{ 1 +
		\sqrt{1 - 4\lambda 	\gamma }
		+ 2\lambda( \mu_H+ \mu_H\sqrt{1 - 4 \lambda \gamma } -2\gamma
        - 4\lambda\mu_H\gamma )} < 	1,
	\end{align*}
	since
	\begin{align*}
	& \mu_H+ \mu_H\sqrt{1 - 4 \red{\lambda} \gamma} -2\gamma -4
	\lambda\mu_H\gamma  > 0
	\end{align*}
	for $\lambda>0$ small enough.
So if $\lambda > 0$ is sufficiently small then Banach's fixed-point theorem yields the existence of a unique $P_{\lambda} \in
	C_{-\gamma} \cap \dom(\Ac_0)$ such that $\Gc(P_{\lambda}) = P_{\lambda}$.
It remains to
	prove that $\{P_{\lambda}\}_{\lambda >  0}$ converges strongly in $\Hi$
	as $\lambda \to 0$ to $P \in C_{-\gamma}\cap \V$ and that the limit $P$
	fulfills
	\begin{align*}
	A^*P+ PA + P^2 = Q.
	\end{align*}
	Applying  $\frac{1}{\lambda}( \lambda \Ac_0 + I )$ to both sides of the
	equation $P_\lambda = \Gc(P_\lambda)$  shows that
	\begin{align*}
	\frac{1}{\lambda} \left( \lambda \Ac_0 + I \right) P_{\lambda} = Q +
	\frac{1}{\lambda}\left(\lambda \Bc + I \right)^{-1} P_{\lambda} ,
	\end{align*}
	which can be rearranged as
	\begin{equation} \label{eq3:AlgApproxEqB}
	\Ac_0 P_{\lambda} + \frac{1}{\lambda}\left(I - \left(\lambda \Bc + I
	\right)^{-1}\right) P_{\lambda}  =  Q.
	\end{equation}
	The nonlinearity in this equation is the Yosida approximation for the
	quadratic term of the Riccati equation. To abbreviate this term, we write in the following
	\begin{align*}
	J_{\lambda} = (\lambda \Bc + I)^{-1}
	\quad
	\text{and}
	\quad
	\Bc_{\lambda}  = \frac{1}{\lambda}(I - J_{\lambda}) = \Bc J_{\lambda}.
	\end{align*}
	Testing \eqref{eq3:AlgApproxEqB} with $P_{\lambda}$, we obtain
	\begin{equation} \label{eq3:algApproxB}
	\ska{\Ac_0 P_{\lambda}}{P_{\lambda}} + \ska{\Bc_{\lambda}
		P_{\lambda}}{P_{\lambda}}  =  \ska{Q}{P_{\lambda}}.
	\end{equation}
	For the first summand on the left-hand side, the estimate
	\begin{align*}
	\ska{\Ac_0 P_{\lambda}}{P_{\lambda}} \geq  2\mu_H\|P_{\lambda}\|_{\Hi}^2
	\end{align*}
	is fulfilled. Using Lemma~\ref{lemma3:BbdNegCoefB} for the second
	summand, it follows that
	\begin{align*}
	\ska{\Bc_{\lambda} P_{\lambda}}{P_{\lambda}}
	&= \frac{1}{\lambda}\ska{\left(I
		- \left(\lambda \Bc + I \right)^{-1}\right) P_{\lambda}}{P_{\lambda}}\\
	&\geq \frac{1}{\lambda}\|P_{\lambda}\|_{\Hi}^2 -
	\frac{1}{\lambda}\|\left(\lambda \Bc + I \right)^{-1} P_{\lambda}\|_{\Hi}
	\|P_{\lambda}\|_{\Hi}\\
	&\geq \frac{1}{\lambda}\left(1  - \frac{ 2}{ 1 + \sqrt{1 - 4\lambda \gamma
		}} \right) \|P_{\lambda}\|_{\Hi}^2.
	\end{align*}
  Inserting these estimates into \eqref{eq3:algApproxB} yields
	\begin{align*}
	\left(2	 \mu_H + \frac{1}{\lambda} \left(1  - \frac{ 2}{ 1 + \sqrt{1 -
			4\lambda \gamma }} \right)
	\right)\|P_{\lambda}\|_{\Hi}
	\leq  \|Q\|_{\Hi}.
	\end{align*}
Applying L'H\^{o}pital's rule shows that
	\begin{align*}
	\frac{1}{\lambda} \left(1  - \frac{ 2}{ 1 + \sqrt{1 - 4\lambda \gamma	}}
	\right) \to -\gamma
	\end{align*}
	as $\lambda \to 0$. Hence, for $c_1 \in (0,2\mu_H-\gamma)$ there
	exists $\varepsilon>0$ such that
	\begin{align*}
	 2\mu_H + \frac{1}{\lambda} \left(1  - \frac{ 2}{ 1 + \sqrt{1 -
			4\lambda 	\gamma }} \right) > c_1
	\end{align*}
	is fulfilled if $\lambda \in (0,\varepsilon)$. Thus,
	\begin{equation}\label{eq3:BoundPLambdaB}
	c_1 \|P_{\lambda}\|_{\Hi}  \leq  \|Q\|_{\Hi}
	\end{equation}
	holds for $\lambda>0$ small enough. This proves that $\{
	P_{\lambda}\}_{\lambda >0 }$ is bounded in $\Hi$ for sufficiently small
	$\lambda >0$. It remains to show that $\{ P_{\lambda}\}_{\lambda >0 }$
	converges in $\Hi$ as $\lambda \to 0$.
  {Using \eqref{eq3:BoundPLambdaB} and Lemma~\ref{lemma3:BbdNegCoefB}, we obtain
  \begin{align}
    \begin{split} \label{eq3:JlambdaConvergence}
      \left\|J_{\lambda} P_{\lambda}-P_{\lambda} \right\|_{\Hi}
      &=\lambda \left\|\frac{1}{\lambda}(I - J_{\lambda}) P_{\lambda}
      \right\|_{\Hi}
      = \lambda \| \Bc( \lambda\Bc + I)^{-1}P_{\lambda} \|_{\Hi}
      \leq \lambda \| (\lambda\Bc + I)^{-1}P_{\lambda} \|_{\Hi}^2\\
      &\leq \frac{4\lambda}{ \left(1+ \sqrt{1-4\lambda\gamma}
        \right)^2} 	\|P_{\lambda}\|_{\Hi}^2
      \leq \frac{4\lambda }{ \left( 1+ \sqrt{1-4\lambda\gamma}
        \right)^2} 	\frac{\|Q\|_{\Hi}^2}{c^2_1}
      \leq c \lambda \|Q\|_{\Hi}^2
    \end{split}
  \end{align}
  for a certain constant $c>0$ if $\lambda>0$ is sufficiently small.
  Now we show the Cauchy property of $\{P_\lambda\}_{\lambda > 0}$.
  As $P_{\lambda}, P_{\nu} \geq -\gamma$ for sufficiently small $\lambda, \nu > 0$, using Lemma~\ref{lemma3:BbdNegCoefB}, it follows that
  \begin{align*}
 J_\lambda P_\lambda  =
    (\lambda \Bc + I)^{-1}P_{\lambda}
    \geq  \frac{ - 2\gamma}{1 + \sqrt{1 - 4\lambda \gamma } }
    \quad \text{and} \quad
    J_\nu P_\nu =
    (\nu \Bc + I)^{-1}P_{\nu}
    \geq  \frac{ - 2\gamma }{ 1 + \sqrt{1 - 4 \nu \gamma }}.
  \end{align*}
  Exploiting these lower bounds and assuming without loss of generality that $\nu  <\lambda$, we obtain that
  \begin{align*}
    &\ska{\Bc_{\lambda} P_{\lambda} - \Bc_{\nu} P_{\nu}}{J_{\lambda}
      P_{\lambda}- J_{\nu}P_{\nu}}\\
    &= \ska{\Bc J_{\lambda}P_{\lambda} - \Bc J_{\nu}P_{\nu}}{J_{\lambda}
      P_{\lambda}- J_{\nu}P_{\nu}}\\
    &\ge \frac{ -4\gamma }{ 1 + \sqrt{1 - 4\lambda \gamma } }
    \|J_{\lambda} P_{\lambda}- J_{\nu} P_{\nu}\|_{\Hi}^2.
  \end{align*}
  Subtracting the equations \eqref{eq3:AlgApproxEqB} for $P_{\lambda}$
  and $P_{\nu}$ and testing with $P_{\lambda} - P_{\nu}$, it follows that
  \begin{align*}
  2\mu_H	\|P_{\lambda} - P_{\nu} \|_{\Hi}^2 \leq - \ska{\Bc_{\lambda}
    P_{\lambda} - \Bc_{\nu}P_{\nu} }{P_{\lambda} - P_{\nu}} .
  \end{align*}
  Altogether this yields
  \begin{align*}
    & \ska{\Bc_{\lambda} P_{\lambda} - \Bc_{\nu} P_{\nu}}{(J_{\lambda}
      P_{\lambda} - P_{\lambda})- (J_{\nu}P_{\nu} - P_{\nu})}\\
    &=\ska{\Bc_{\lambda} P_{\lambda} - \Bc_{\nu} P_{\nu}}{J_{\lambda}
      P_{\lambda}- J_{\nu}P_{\nu}}
    - \ska{\Bc_{\lambda}
      P_{\lambda} - \Bc_{\nu}P_{\nu} }{P_{\lambda} - P_{\nu}} \\
    &\geq 2\mu_H	\|P_{\lambda} - P_{\nu} \|_{\Hi}^2
    - \frac{ 4\gamma }{ 1 + \sqrt{1 - 4\lambda \gamma } }
    \|J_{\lambda} P_{\lambda}- J_{\nu} P_{\nu}\|_{\Hi}^2\\
    & \geq 2\mu_H	\|P_{\lambda} - P_{\nu} \|_{\Hi}^2
    - \frac{ 4\gamma }{ 1 + \sqrt{1 - 4\lambda \gamma } }
    \big( \|J_{\lambda} P_{\lambda}- P_{\lambda}\|_{\Hi}
    + \| P_{\lambda}- P_{\nu}\|_{\Hi}
    + \| P_{\nu}- J_{\nu} P_{\nu}\|_{\Hi} \big)^2\\
    & = \Big( 2\mu_H - \frac{ 4\gamma }{ 1 + \sqrt{1 - 4\lambda \gamma }
    }\Big)	\|P_{\lambda} - P_{\nu} \|_{\Hi}^2\\
    &\quad- \frac{ 4\gamma }{ 1 + \sqrt{1 - 4\lambda \gamma } }
    \big( \|J_{\lambda} P_{\lambda}- P_{\lambda}\|_{\Hi}^2
    + \| P_{\nu}- J_{\nu} P_{\nu}\|_{\Hi}^2
    + \|J_{\lambda} P_{\lambda}- P_{\lambda}\|_{\Hi} \| P_{\lambda}-
    P_{\nu}\|_{\Hi}\\
    &\quad \phantom{- \frac{ 4\gamma }{ 1 + \sqrt{1 - 4\lambda \gamma }
    }\big( \| }+
    \| P_{\lambda}- P_{\nu}\|_{\Hi} \| P_{\nu}- J_{\nu} P_{\nu}\|_{\Hi}
    + \|J_{\lambda} P_{\lambda}- P_{\lambda}\|_{\Hi} \| P_{\nu}- J_{\nu}
    P_{\nu}\|_{\Hi} \big).
  \end{align*}
  On the other hand, we have the upper bound
  \begin{align*}
    & \ska{\Bc_{\lambda} P_{\lambda} - \Bc_{\nu} P_{\nu}}{(J_{\lambda}
      P_{\lambda} - P_{\lambda})- (J_{\nu}P_{\nu} - P_{\nu})}\\
    & \leq \| \Bc_{\lambda} P_{\lambda} - \Bc_{\nu} P_{\nu} \|_{\Hi}
      \big( \| J_{\lambda} P_{\lambda} - P_{\lambda} \|_{\Hi}
      + \| J_{\nu}P_{\nu} - P_{\nu} \|_{\Hi}  \big).
  \end{align*}
  Both together imply that
  \begin{align*}
    &\Big( 2\mu_H - \frac{ 4\gamma }{ 1 + \sqrt{1 - 4\lambda \gamma }
    }\Big)	\|P_{\lambda} - P_{\nu} \|_{\Hi}^2\\
    & \leq \| \Bc_{\lambda} P_{\lambda} - \Bc_{\nu} P_{\nu} \|_{\Hi}
    \big( \| J_{\lambda} P_{\lambda} - P_{\lambda} \|_{\Hi}
    + \| J_{\nu}P_{\nu} - P_{\nu} \|_{\Hi}  \big) \\
    &\quad + \frac{ 4\gamma }{ 1 + \sqrt{1 - 4\lambda \gamma } }
    \big( \|J_{\lambda} P_{\lambda}- P_{\lambda}\|_{\Hi}^2
    + \| P_{\nu}- J_{\nu} P_{\nu}\|_{\Hi}^2
    + \|J_{\lambda} P_{\lambda}- P_{\lambda}\|_{\Hi} \| P_{\lambda}-
    P_{\nu}\|_{\Hi}\\
    &\quad \phantom{- \frac{ 4\gamma }{ 1 + \sqrt{1 - 4\lambda \gamma }
      }\big( \| }+
    \| P_{\lambda}- P_{\nu}\|_{\Hi} \| P_{\nu}- J_{\nu} P_{\nu}\|_{\Hi}
    + \|J_{\lambda} P_{\lambda}- P_{\lambda}\|_{\Hi} \| P_{\nu}- J_{\nu}
    P_{\nu}\|_{\Hi} \big).
  \end{align*}
  Since $ \gamma < \mu_H $, for $c_2 \in (2\gamma,
  2\mu_H)$, we have
  \begin{align*}
    2\mu_H - \frac{ 4\gamma }{ 1 + \sqrt{1 - 4\lambda \gamma }
    } > c_2
  \end{align*}
  if $\lambda >0$ sufficiently small.
 Employing \eqref{eq3:BoundPLambdaB} and
  \eqref{eq3:JlambdaConvergence}, this proves
  \begin{align*}
    c_2 \|P_{\lambda} - P_{\nu} \|_{\Hi}^2 \to 0 \quad \text{as } \lambda, \nu \to
    0.
  \end{align*}
	Thus, there exists $P \in \Hi$ such that  $P_{\lambda} \to P$ in $\Hi$
	as $\lambda \to 0$. Since  $C_{-\gamma}$ is closed, we find that $ P\in C_{-\gamma}$.}

	It remains to prove that $A^*P + PA +P^2 = Q$ is satisfied and $P \in \V$.
	Because of \eqref{eq3:JlambdaConvergence}, it follows that
	\begin{align*}
	J_{\lambda}  P_{\lambda} \to P  \quad \text{ in } \Hi \text{ as } \lambda \to 0.
	\end{align*}
	Since $\Bc$ is continuous, we obtain that
	\begin{align*}
	\Bc_{\lambda} (P_{\lambda}) = \Bc J_{\lambda} (P_{\lambda}) \to \Bc P =
	P^2 \quad \text{ in } \Hi \text{ as } \lambda \to 0.
	\end{align*}
	For the linear part, it can be concluded that
	\begin{align*}
	\Ac_0 P_{\lambda} = Q - \Bc_{\lambda} (P_{\lambda}) \to Q - P^2 \quad
	\text{ in } \Hi \text{ as } \lambda \to 0.
	\end{align*}
Since
	\begin{align*}
	\mu_V \|P_{\lambda} \|_{\V}^2 \leq \ska{\Ac_0 P_{\lambda}}{P_{\lambda}}
	\leq \|\Ac_0P_{\lambda} \|_{\Hi} \|P_{\lambda}\|_{\Hi},
	\end{align*}
and since $\{ \Ac_0P_{\lambda} \}_{\lambda >0 }$ and $\{ P_{\lambda}\}_{\lambda
	>0 }$ are convergent and thus bounded in $\Hi$, we see that
$ \{ P_{\lambda}\}_{\lambda >0 }$ is also bounded in $\V$. Therefore, there
	exists a weakly convergent subsequence in $\V$. The uniqueness of the
	limit implies that this limit has to be $P$ and that $P$ is an element of
	$\V$. Since $\Ac_0 :\V \to \V^*$ is a linear and bounded operator, it is weakly-weakly continuous, see
	\cite[Theorem~3.10]{H.Brezis.2010}. Therefore, $\{ \Ac_0P_{\lambda}
	\}_{\lambda >0 }$ converges weakly  to $\Ac_0 P$ in $\V^*$ and the
	convergence of $\{ \Ac_0P_{\lambda} \}_{\lambda >0 }$ in $\Hi$ implies that
	\begin{align*}
	\Ac_0 P_{\lambda} \to \Ac_0 P \quad \text{ in } \V^* \text{ as } \lambda \to 0
	\end{align*}
	which yields that $\Ac_0 P + \Bc P = Q$.
\end{proof}

\section{Weak solution of the Riccati equation}
\label{section:varSol}

In this section, we consider the Hilbert spaces $V$ and $H$ as stated in
Assumption~\ref{as:spaces}. Further, we introduce the following mappings.

\begin{as} \label{as:OpAB}
	Let $a : [0, T ] \times V \times V \to \R$ be given such that for every $t
	\in [0, T ]$ the form $a(t; \cdot, \cdot) : V
	\times 	V \to	\R$ is bilinear. There exist constants $\mu, \eta > 0$
	such that for every $u, v \in V$ and every $t \in [0,T]$
	\begin{alignat*}{2}
	a(t; u, u) &\geq \mu \|u\|^2_{V},
	\quad
	&&\text{(i.e. $a$ is \emph{uniformly strongly positive}),}\\
	|a(t; u, v)| &\leq \eta \|u\|_{V} \|v\|_{V},
  \quad
  &&\text{(i.e. $a$ is \emph{uniformly bounded})}.
	\end{alignat*}
	Further, let $a(\cdot; u, v): [0,T] \to \R$ be Lebesgue-measurable for fixed
	$u,	v \in V$.
\end{as}

This assumption in mind, for every $t \in[0, T ]$ we introduce the operators
$\A(t), \A^*(t) : V \to
V^*$  given by
\begin{equation*}
\dual[V^*\times V]{\A(t)u}{v}= a(t; u, v),
\quad \text{ and } \quad
\dual[V^*\times V]{\A^*(t)u}{v}= a(t; v, u),
\end{equation*}
with $u, v \in V$. Note that $\A^*(t)$ is the dual operator of $\A(t)$ for
every $t \in [0,T]$.
\red{In our setting, we need to assume that $\A(t)$, $t
  \in [0,T]$, is linear and strongly positive. This rather restrictive assumption
  is necessary to allow semidefinite data but can be fulfilled, for example, by
  elliptic differential operators.}

We consider the initial value problem \eqref{eq0:InitialRic} for the Riccati differential equation
and we use the backward Euler scheme
to obtain a time discretization. The resulting semi-discrete problem can be
solved using the existence result for algebraic Riccati equations from
Section~\ref{section:algRic}.\:This approach will lead to the following
result, which we then prove in detail.

\begin{satz} \label{satz4:variationelLoesung}
	Let Assumptions~\ref{as:spaces}.\:and \ref{as:OpAB}.\:be
	satisfied.
	For $\gamma$ with $\red{0\le \gamma < \frac{ \mu }{C_{V,H}^2}}$, $\Qc \in
	L^{1}(0,T;\Hi) +	L^{2}(0,T;\V^*)$ with
	$\Qc(t)=\Qc^*(t) \in \Hi$ and $\Qc(t)\geq -\gamma^2  $ for
	almost every $t\in [0,T]$ as well as $\Pc_0 \in \Hi$ with $\Pc_0=\Pc_0^*$ and
	$\Pc_0\geq - \gamma $, there exists a weak solution $\Pc \in
	\W$	to the initial value problem \eqref{eq0:InitialRic} \red{that fulfills $\Pc(t)
	\geq - \gamma$ for every $t \in [0,T]$. This solution is unique.}
\end{satz}

Again, let us remark that we construct the maximal solution.

\subsection{Time discrete problem}

For the time discretization, we consider the equidistant partition $0 =t_0 <
\dots < t_N = T$ with $\tau = \frac{T}{N}$ and $t_n = n\tau$ ($
n=0,1,\dots,N$). We always assume that $\tau$ is sufficiently small such that $\tau < \frac{\mu}{2C_{V,H}^2}$.
We consider
the semi-discrete problem
\begin{equation} \label{eq4:RiccSemi}
\frac{P_n-P_{n-1}}{\tau} + \Ac_n P_{n} = Q_n,\quad n=1, 2, \dots, N,
\end{equation}
with $P_0 = \Pc_0$, where $P_n \in \V$ ($n=1, 2, \dots , N$) denotes an approximation of $\Pc(t_n)$.
Here the right-hand side is given by $Q_n
= \frac{1}{\tau} \int_{t_{n-1}}^{t_n} \Qc(t) dt \in \Hi$. Moreover, we define
$\Ac_n P_n = A_n^*P_n + P_nA_n + P_n^2 \in \V^*$ with
\begin{align*}
A_n = \frac{1}{\tau} \int_{t_{n-1}}^{t_n} \A(t) dt
\end{align*}
for $n=1,2,\dots, N$.
In order to show the existence of a solution $(P_n)_{n=1}^N$ to the
semi-discrete problem \eqref{eq4:RiccSemi}, we make use of Theorem~\ref{satz3:AlgRicNeg}.
We first consider a somewhat more regular right-hand side $\Qc \in
L^2(0,T;\Hi)$.

\begin{lemma} \label{lemma4:Apriori}
	Let Assumptions~\ref{as:spaces}.\:and \ref{as:OpAB}.\:be
	satisfied and let $\tau < \frac{\mu}{2C_{V,H}^2}$.
	For $\gamma$ with $0\le \gamma<  \frac{\mu }{C_{V,H}^2}$, $\Qc \in
	L^{2}(0,T;\Hi)$ with
	$\Qc(t)=\Qc^*(t)$,  $\Qc(t)\geq - \gamma^2$ for
	almost every $t\in [0,T]$,  $\Pc_0 \in \Hi$ with $\Pc_0=\Pc_0^*$, and
	$\Pc_0\geq - \gamma$,
	the semi-discrete problem \eqref{eq4:RiccSemi} admits a solution
	$(P_n)_{n=1}^N$ such that $P_n \in \V$ with $P_n \geq	-\gamma$, $n=1, 2, \dots,N$, that
	fulfills the a priori estimate

	\begin{align} \label{eq4:aPrioriEst}
	\begin{split}
  &\|P_n \|^2_{\Hi} +\sum_{k=1}^{n} \|P_k - P_{k-1}\|^2_{\Hi}
	+ \left( \mu - 	\frac{\gamma C_{V,H}^2 }{2}\right)
	\tau\sum_{k=1}^{n} \|P_k\|^2_{\V} \\
	&\quad \leq   \|\Pc_{0}\|^2_{\Hi}
	+\frac{1}{\mu - \frac{\gamma C_{V,H}^2 }{2}}\|\Qc\|_{L^2(0,T;\V^*)}^2, \quad
	n = 1, 2, \dots, N.
	\end{split}
	\end{align}
\end{lemma}


\begin{proof}
	Considering the semi-discrete problem \eqref{eq4:RiccSemi}
	for each step $n=1, 2, \dots , N$, we obtain an algebraic Riccati equation of the form
	\begin{align}\label{eq4:ricAlg}
	\left(A_n + \frac{1}{2\tau}I \right) ^* P_n + P_n \left(A_n +
	\frac{1}{2\tau}I
	\right) + P_n^2 = Q_n + \frac{1}{\tau} P_{n-1}
	\end{align}
	with
	\begin{align*}
	\dual{\left(A_n + \frac{1}{2\tau}I \right) P }{P}
	\geq \mu \|P\|^2_{\V} + \frac{1}{2\tau} \|P\|^2_{\Hi}
	\geq \left(\frac{ 2 \mu}{C_{V,H}^2 }   +
	\frac{1}{2\tau}\right) \|P\|^2_{\Hi}
	\end{align*}
	for $P\in \V$.
	We can then apply Theorem~\ref{satz3:AlgRicNeg} (choosing $\gamma$ in
Theorem~\ref{satz3:AlgRicNeg} appropriately with $\mu_H := \frac{ 2 \mu}{C_{V,H}^2 }   +
	\frac{1}{2\tau}$) if
	\begin{align*}
	Q_n + \frac{1}{\tau} P_{n-1}
	\geq - \left(\gamma + \frac{1}{2\tau}\right)^2
	> - \left( \frac{ 2 \mu}{C_{V,H}^2 }  +
	\frac{1}{2\tau}\right)^2
	= - \mu_H^2
	\end{align*}
	is fulfilled for every $n=1, 2, \dots,N$. To prove this condition, we argue
	inductively. Since $P_{0} \geq- \gamma$ and $Q_1 \geq - \gamma^2 $, we
	obtain that
	\begin{align*}
	Q_1 + \frac{1}{\tau} P_{0}
	\geq - \gamma^2 - \frac{\gamma}{\tau}
	> - \left(\left(\frac{ 2 \mu}{C_{V,H}^2 }
	\right)^2 + \frac{1}{\tau }  \frac{ 2 \mu}{C_{V,H}^2 }  +
	\frac{1}{4\tau^2}\right)
	= - \left(\frac{ 2 \mu}{C_{V,H}^2 }  + \frac{1}{2\tau}\right)^2.
	\end{align*}
	The existence of $P_1 \in \V$ then follows  using
	Theorem~\ref{satz3:AlgRicNeg}.
	It remains to prove that $P_{n-1}\geq -\gamma$ implies
	$P_n \geq -\gamma$ for arbitrary $n=1,2,\dots,N$.
	Since the existence of a compact and self-adjoint operator $P_n$ already
	follows from
	using $P_{n-1} \geq -\gamma$, there exists an orthonormal
	system
	$\seqd{e}{i}$ of eigenvectors in $H$ and real eigenvalues
	$\seqd{\alpha}{i}$ in $\R$ such that
	\begin{align*}
	P_n = \sum_{i=1}^{\infty} \alpha_i \ska[H]{e_i}{\cdot} e_i.
	\end{align*}
	Note that the eigenvectors and eigenvalues depend on $n$ but to
	keep the notation simple, we will not state this dependence.
  Testing \eqref{eq4:ricAlg} with $e_i$, $i\in \N$, it follows that
	\begin{align*}
	\ska[H]{P_n^2 e_i}{e_i} + 2\ska[H]{\left(A_n +
		\frac{1}{2\tau}I\right)^*
		P_n 	e_i}{e_i}  - \ska[H]{Q_n e_i}{e_i} - \frac{1}{\tau} \ska[H]{P_{n-1}
		e_i}{e_i}
	=  0.
	\end{align*}
	Abbreviating the terms $\ska[H]{Q_n e_i}{e_i} = q_i$,	 $\ska[H]{P_{n-1}
	e_i}{e_i} = p_i$,	and $\dual[V^*\times V]{A_n e_i}{e_i} = \mathbf{a}_i$, the
	equation can be simplified to
	\begin{align*}
	\alpha_i^2 + 2\alpha_i \left(\mathbf{a}_i + \frac{1}{2\tau} \right)  -
	q_i - \frac{p_i}{\tau}	 =  0.
	\end{align*}
Note that the discriminant is larger than $\frac{1}{4\tau^2}$ so that the roots are real. Since we know from Theorem~\ref{satz3:AlgRicNeg} (choosing $\gamma$ in Theorem~\ref{satz3:AlgRicNeg} appropriately) that $\alpha_i >
- \left(\frac{2\mu}{C_{V,H}^2} + \frac{1}{2\tau}\right)$, we only have to consider the larger of the two solutions of this quadratic equation
in $\R$ given by
	\begin{align*}
	&- \left( \mathbf{a}_i  + \frac{1}{2\tau} \right)
	+
	\sqrt{\left(\mathbf{a}_i + \frac{1}{2\tau }
		\right)^2 + q_i + \frac{p_i}{\tau } }
  = \frac{q_i + \frac{1}{\tau} p_i}{ \mathbf{a}_i + \frac{1}{2\tau}
  + \sqrt{\left(\mathbf{a}_i + \frac{1}{2\tau} \right)^2 +  q_i +
		\frac{p_i}{\tau}}}
	\end{align*}
if $\tau$ is sufficiently small such that $\tau < \frac{C_{V,H}^2}{2\mu}$.

Due to Assumption~\ref{as:OpAB}, the values $\mathbf{a}_i$ fulfill the
	estimate
	$\mathbf{a}_i\geq \mu \|e_i\|_V^2 \geq \frac{\mu}{C_{V,H}^2} $
	for every $i\in \N$. Since we have assumed that $P_{n-1} \geq 	-\gamma$
	and $Q_{n-1} \geq-\gamma^2$, it follows that $p_i \geq
	-\gamma$ and $q_i \geq -\gamma^2$.
	Therefore, we find
	\begin{align*}
	\alpha_i
	&= \frac{q_i + \frac{1}{\tau} p_i}{ \mathbf{a}_i + \frac{1}{2\tau} +
		\sqrt{\left(\mathbf{a}_i + \frac{1}{2\tau} \right)^2 +  q_i +
		\frac{1}{\tau} p_i}}\\
	&\geq  \frac{-\gamma^2  - \frac{1}{\tau} \gamma }{ \mathbf{a}_i +
		\frac{1}{2\tau}  + \sqrt{ \mathbf{a}_i^2 + \frac{1}{\tau} \mathbf{a}_i+
			\frac{1}{4\tau^2} +  q_i + \frac{1}{\tau}
			p_i}}\\
	&\geq \frac{-\gamma\left(\gamma  + \frac{1}{\tau}\right)  }{
		\mathbf{a}_i +
		\frac{1}{2\tau}  + \sqrt\frac{1}{4\tau^2}}
	\geq \frac{-\gamma\left(\frac{\mu}{C_{V,H}^2}   +
	\frac{1}{\tau}\right)
	}{ \frac{\mu}{C_{V,H}^2}  + \frac{1}{\tau} }
	= -\gamma,
	\end{align*}
	where we have employed that
	\begin{align*}
		\mathbf{a}_i +  p_i
		> \gamma - \gamma
		\geq 0
	\quad \text{and} \quad
	\mathbf{a}_i^2 +  q_i
	> \gamma^2 - \gamma^2
	\geq 0.
	\end{align*}
	This proves that $P_n\geq - \gamma$ and thus the
	existence of all $P_n$ for $n=1,2,\dots,N$.

	Let us now derive an a priori bound. Testing \eqref{eq4:RiccSemi}
	with $P_n$, and using both
  \begin{align*}
  \ska{P_n - P_{n-1}}{P_n}
  = \frac{1}{2} \|P_n \|^2_{\Hi} - \frac{1}{2} \|P_{n-1}\|^2_{\Hi}
  + \frac{1}{2} \|P_n -  P_{n-1}\|^2_{\Hi}
  \end{align*}
  and Young's inequality,
 we obtain  that
	\begin{equation}\label{eq4:NegAprioriTested}
	\begin{split}
	&\frac{1}{2\tau}\left(\|P_n \|^2_{\Hi} - \|P_{n-1}\|^2_{\Hi} + \|P_n -
	P_{n-1}\|^2_{\Hi} \right) + \mu \|P_n\|^2_{\V} + \ska{P_n^2}{P_n}\\
	& \leq \frac{1}{2 \left(\mu - \frac{\gamma C_{V,H}^2 }{2}
		\right)} \|Q_n\|_{\V^*}^2 + \frac{\mu -
		\frac{\gamma C_{V,H}^2 }{2} }{2	}
	\|P_n\|^2_{\V}	.
	\end{split}
	\end{equation}
	Using Lemma~\ref{lemma3:QuadTermNeg},
	we can estimate the nonlinearity by
	\begin{align*}
	\ska{P_n^2}{P_n}
	\geq -\gamma \|P_n \|^2_{\Hi} \geq - \frac{\gamma C_{V,H}^2
	}{2} \|P_n\|^2_{\V}.
	\end{align*}
	Inserting this estimate into \eqref{eq4:NegAprioriTested}, it follows that
	\begin{align*}
	&\frac{1}{2\tau}(\|P_n \|^2_{\Hi} - \|P_{n-1}\|^2_{\Hi} + \|P_n -
	P_{n-1}\|^2_{\Hi}) + \left( \mu -  \frac{\gamma C_{V,H}^2
	}{2} \right) \|P_n\|^2_{\V} \\
	&\quad \leq \frac{1}{2 \left( \mu- \frac{\gamma C_{V,H}^2
		}{2} \right)} \|Q_n\|_{\V^*}^2 + \frac{\mu -
		\frac{\gamma C_{V,H}^2 }{2}	}{2	} \|P_n\|^2_{\V}
	\end{align*}
	and, therefore,
	\begin{align*}
	\|P_n \|^2_{\Hi} - \|P_{n-1}\|^2_{\Hi} + \|P_n - P_{n-1}\|^2_{\Hi} +
	\left( \mu- \frac{\gamma C_{V,H}^2
	}{2} \right)\tau \|P_n\|^2_{\V} \leq
	\frac{\tau}{\mu-\frac{\gamma C_{V,H}^2 }{2} }
	\|Q_n\|_{\V^*}^2 .
	\end{align*}
	Summing up leads to the estimate
	\begin{align}
	\begin{split}\label{eq4:apriori}
	&\|P_n \|^2_{\Hi} +\sum_{k=1}^{n} \|P_k - P_{k-1}\|^2_{\Hi} +
	\left(\mu-\frac{\gamma C_{V,H}^2 }{2} \right)\tau\sum_{k=1}^{n}
	\|P_k\|^2_{\V} \\
	&\quad\leq \|P_{0}\|^2_{\Hi} + \frac{\tau}{\mu-\frac{\gamma C_{V,H}^2 }{2}}
	\sum_{k=1}^{N} \|Q_k\|_{\V^*}^2, \quad n = 1, 2, \dots, N .
	\end{split}
	\end{align}
	Since
	\begin{align}\label{eq4:boundQ}
	\tau \sum_{i=1}^{N} \|Q_n\|_{\V^*}^2
	\leq  \sum_{i=1}^{N} \int_{t_{i-1}}^{t_i} \|\Qc(t)\|_{\V^*}^2 dt
	= \|\Qc\|_{L^2(0,T;\V^*)}^2,
	\end{align}
	the right-hand side of \eqref{eq4:apriori} can be simplified and we obtain
	the desired a priori estimate \eqref{eq4:aPrioriEst}, recalling that $\Pc_0 = P_0$.
\end{proof}

Using the solution $\left( P_n\right)_{n=1}^{N}$ of the semi-discrete
problem \eqref{eq4:RiccSemi}, we define both a piecewise constant and a
piecewise linear interpolation. For $t \in (t_{n-1},t_{n}]$,
$n=1,2,\dots,N$, define
\begin{align} \label{eq4:defInterpol}
\Pc_{\tau} (t) = P_{n}
\quad
\text{and}
\quad
\Pctil_{\tau} (t) = \frac{P_{n} - P_{n-1}}{\tau} (t-t_{n-1}) + P_{n-1},
\end{align}
with $\Pctil_{\tau} (0) = \Pc_{\tau} (0) = P_0$. Further, we define the
piecewise constant interpolations for the discrete right-hand side $(Q_n)_{n=1}^N$ and
for $(\Ac_n)_{n=1}^N$: For $t \in (t_{n-1},t_{n}]$,
$n=1,2,\dots,N,$ we define
\begin{align*}
\Qc_{\tau} (t) = Q_{n}, \quad \A_{\tau}(t) = A_{n}
, \quad
\Ac_{\tau}(t)P
= \A^*_{\tau}(t)P + P\A_{\tau}(t) +P^2
\end{align*}
for $P \in \V$.
As the function $\Pctil_{\tau}: [0,T] \to \Hi$ is piecewise linear and
continuous, it is weakly differentiable with the derivative $\Pctil_{\tau}'(t)
=\frac{P_{n} - P_{n-1}}{\tau}$ for $t \in (t_{n-1},t_{n})$, $n= 1,2, \dots , N$. Using the
functions $ \Pc_{\tau} $ and $\Pctil_{\tau}$, the semi-discrete problem
\eqref{eq4:RiccSemi} reads
\begin{equation} \label{eq4:RiccDiscret}
\begin{split}
\Pctil'_{\tau}(t) + \Ac_{\tau}(t) \Pc_{\tau}(t) &=\Qc_{\tau}(t) \quad
\text{ for almost every } t\in (0,T),\\
\Pc_{\tau} (0) &= P_{0}.
\end{split}
\end{equation}

\subsection{Limiting process} \label{subsection:limit}

The approximate solutions $\Pctil_{\tau}$ and $\Pc_{\tau}$
from the last subsection provide a sequence of
functions that will be shown to converge to a weak solution of
\eqref{eq0:InitialRic} as $\tau \to 0$. In this subsection, we examine each term of
\eqref{eq4:RiccDiscret} for the limiting process $\tau \to 0$.

\begin{lemma} \label{lemma4:konvTF}
	Let Assumptions~\ref{as:spaces}.\:and \ref{as:OpAB}.\:be
	satisfied.
	Further let  $0\le \gamma <  \frac{\mu }{C_{V,H}^2}$, $\Qc \in
	L^{2}(0,T;\Hi)$ with
	$\Qc(t)=\Qc^*(t)$ and $\Qc(t)\geq - \gamma^2$ for
	almost every $t\in [0,T]$ and $\Pc_0 \in \Hi$ with $\Pc_0=\Pc_0^*$, and
	$\Pc_0\geq -\gamma$.
	Let $(N_k)_{k\in\N}$ be a sequence of positive integers such that
	$N_k \to \infty$ as $k \to \infty$ and let $(\tau_k)_{k\in\N}$ be the
	sequence of step sizes $\tau_k = \frac{T}{N_k}$ such that $\sup_{k\in\N}\tau_k < \frac{\mu}{2C_{V,H}^2}$. Then there exists a
	subsequence $(N_{k'})_{k'\in \N}$ and $\Pc \in L^{\infty}(0,T;\Hi)
\cap L^2(0,T;\V)$ with $\Pc' \in L^2(0,T;\V^*)$ (so that $\Pc \in \W$)
such that the
	sequences $(\Pc_{\tau_{k'}})_{k' \in \N}$ and $(\Pctil_{\tau_{k'}})_{k' \in
	\N}$ of approximate solutions to \eqref{eq0:InitialRic} satisfy
	\begin{align*}
	\Pc_{\tau_{k'}} &\weak \Pc \text{ in } L^2(0,T;\V),\\
	\Pctil_{\tau_{k'}},\Pc_{\tau_{k'}} &\weaks \Pc \text{ in }
	L^{\infty}(0,T;\Hi), \quad \text{ and }\\
	\Pctil'_{\tau_{k'}} &\weak \Pc' \text{ in } L^2(0,T;\V^*)
	\end{align*}
	as $k' \to \infty$.
\end{lemma}
\begin{proof}
	In the following, we will omit the indices $k$  to keep the notation simple.
	Using the a priori estimate \eqref{eq4:aPrioriEst}, we obtain the uniform
	boundedness of the piecewise constant and the piecewise linear
	interpolation with respect to the norm of $L^{\infty}(0,T;\Hi)$:
	\begin{align*}
	\|\Pc_{\tau}\|_{L^{\infty}(0,T;\Hi)} \leq c , \quad
	\|\Pctil_{\tau} 	\|_{L^{\infty}(0,T;\Hi)} \leq c,
	\end{align*}
	where $c>0$ only depends on the data of the problem but not on $\tau$.
	In an analogous manner, \eqref{eq4:aPrioriEst}
	yields the uniform boundedness of the piecewise constant interpolation in
	$L^{2}(0,T;\V)$,
	\begin{align*}
	\|\Pc_{\tau}\|_{L^{2}(0,T;\V)}^2
	= \sum_{n=1}^{N} \int_{t_{n-1}}^{t_n}\|\Pc_{\tau}(t)\|_{\V}^2 dt
	= \tau \sum_{n=1}^{N}\|P_{n}\|_{\V}^2  \leq c.
	\end{align*}
	For the $L^2(0,T;\V^*)$-norm of the derivative of $\Pctil_{\tau}$, it
	follows that
	\begin{align*}
	\|\Pctil_{\tau}' \|_{L^2(0,T;\V^*)}^2
	&= \sum_{n=1}^{N} \int_{t_{n-1}}^{t_n} \left\|\frac{P_{n} -
		P_{n-1}}{\tau}\right\|^2_{\V^*} dt
	=  \tau \sum_{n=1}^{N} \left\|Q_n - \Ac_nP_n \right\|^2_{\V^*} \\
	&\leq  2 \tau \sum_{n=1}^{N} \left(\|Q_n\|^2_{\V^*} + \| \Ac_nP_n
	\|^2_{\V^*}\right) .
	\end{align*}
  Moreover, we recall \eqref{eq4:boundQ} and, using
  Assumption~\ref{as:OpAB}, we see that
  \begin{align*}
    \| \Ac_n P_n \|_{\V^*}
    &\leq \| A_n^* P_n \|_{\V^*} + \| P_n A_n \|_{\V^*} + \| P_n^2 \|_{\V^*}
    \\
    &\leq 2 \eta \| P_n \|_{\V} + \frac{C_{V,H}}{\sqrt{2}} \| P_n^2 \|_{\Hi}
    \leq 2 \eta \| P_n \|_{\V} + \frac{C_{V,H}}{\sqrt{2}} \| P_n \|_{\Hi}^2,
  \end{align*}
  and thus
  \begin{align*}
    \tau \sum_{n=1}^{N} \|\Ac_n P_n \|_{\V^*}^2
    \leq c \Big( \tau \sum_{n=1}^{N} \| P_n \|_{\V}^2 + \max_{n=1,2,\dots,N} \| P_n
    \|_{\Hi}^4 \Big)
    \leq c.
  \end{align*}
  Because of the reflexivity of $L^2(0,T;\V)$ and its dual and the fact that
	$L^{\infty}(0,T;\Hi)$ is the dual of the separable normed space
	$L^{1}(0,T;\Hi)$,
	there exist $\Pc \in L^{\infty}(0,T;\Hi) \cap
	L^2(0,T;\V)$, $\Pctil \in L^{\infty}(0,T;\Hi)$ and $\Rc \in L^2(0,T;\V^*)$
	such that
	\begin{align*}
	\Pc_{\tau} &\weaks \Pc \text{ in } L^{\infty}(0,T;\Hi),\\
	\Pctil_{\tau} &\weaks \Pctil \text{ in } L^{\infty}(0,T;\Hi),\\
	\Pc_{\tau} &\weak \Pc \text{ in } L^{2}(0,T;\V), \quad \text{ and }\\
	\Pctil'_{\tau} &\weak \Rc \text{ in } L^{2}(0,T;\V^*)
	\end{align*}
	as $\tau \to 0$ (passing to a subsequence if necessary), compare
	\cite[Chapter 3]{H.Brezis.2010} for further details. The next step is
	to prove that the limit $\Pc$ of the constant interpolation coincides with
	the limit $\Pctil$ of the linear interpolation. The a priori estimate
	\eqref{eq4:aPrioriEst} implies that
	\begin{align*}
	\int_{0}^{T}\| \Pc_{\tau}(t) - \Pctil_{\tau}(t) \|_{\Hi} ^2 dt
	&= \sum_{n=1}^{N} \int_{t_{n-1}}^{t_n}  \left\| \frac{P_{n}- P_{n-1}}{\tau}
	(t_{n}-t) \right\|_{\Hi} ^2 dt\\
	&= \frac{\tau}{3} \sum_{n=1}^{N} \left\| P_{n}- P_{n-1} \right\|_{\Hi}^2
	\leq \tau c \to 0
	\end{align*}
	as $\tau \to 0$.
	Therefore, the limits of $(\Pc_{\tau})$ and
	$(\Pctil_{\tau})$ coincide. The last step to prove the assertion is
	to show that $\Pc' =\Rc$. For arbitrary $S \in \V$ and $\varphi \in
	C^{\infty}_c(0,T)$, it follows that
	\begin{align*}
	\int_{0}^{T} \dual{\Rc(t)}{S} \varphi(t) dt
	&= \int_{0}^{T} \dual{\Rc(t)- \Pctil'_{\tau}(t) }{S} \varphi(t) dt -
	\int_{0}^{T} \dual{\Pctil_{\tau}(t)}{S} \varphi'(t) dt\\
	&\to - \int_{0}^{T} \dual{\Pc(t)}{S} \varphi'(t) dt
	\end{align*}
	as $\tau \to 0$. Thus, $\Rc \in L^2(0,T;\V^*)$ is the weak derivative
	$\Pc'$ of $\Pc \in L^2(0,T;\V) \subseteq L^2(0,T;\V^*)$. This finally shows that $\Pc \in \W$.
\end{proof}

The next step is to study the behaviour of the operators $\A_{\tau}^*(t)$ and
$\A_{\tau}(t) $ as $\tau \to 0$.

\begin{lemma} \label{lemma4:KonvergenzA}
	Let Assumptions~\ref{as:spaces}.\:and \ref{as:OpAB}.\:be satisfied.
  Further, let $(N_k)_{k\in\N}$ be a sequence of positive integers such that
	$N_k \to \infty$ as $k \to \infty$ and let $(\tau_k)_{k\in\N}$ be the
	sequence of step sizes $\tau_k = \frac{T}{N_k}$ such that $\sup_{k\in\N}\tau_k < \frac{\mu}{2C_{V,H}^2}$.
	Then for all 	$P \in \V$ both
	\begin{align*}
	\A_{\tau_k}^* (t) P \to  \A^*(t) P \text{ in } \V^*
	\quad
	\text{and}
	\quad
	P \A_{\tau_k} (t) \to  P \A(t) \text{ in } \V^*
	\end{align*}
	hold for almost every $t\in [0,T]$ as $k\to\infty$. Furthermore, also
  \begin{align*}
  \A_{\tau_k}^* P \to  \A^* P \text{ in } L^r(0,T;\V^*)
  \quad
  \text{and}
  \quad
  P \A_{\tau_k} \to  P \A \text{ in } L^r(0,T;\V^*)
  \end{align*}
  for every $r \in [1,\infty)$ as $k \to \infty$.
\end{lemma}

{
We omit the proof of Lemma~\ref{lemma4:KonvergenzA}. The assertion can
be verified employing that both $t \mapsto \A(t) P$ and $t \mapsto P \A(t)$
are elements of $L^r(0,T;\V^*)$ for every $r \in [1,\infty)$ and the fact that
$C^{\infty}([0,T];\V^*)$ is a dense subspace of $L^r(0,T;\V^*)$. See, e.g.,
\cite[Remark~8.15 and Remark~8.21]{T.Roubicek.2013} for more details.
Using that the space $C^{\infty}([0,T];\Hi)$ is dense in
$L^1(0,T;\Hi) + L^2(0,T;\V^*)$, the same argumentation yields the strong
convergence of $\Qc_{\tau}$ to $\Qc$ in $L^1	(0,T;\Hi) + L^2(0,T;\V^*)$ as
$\tau \to 0$.}

\subsection{Weak solution of the Riccati equation}

Using the results of the previous subsections, the assertion of
Theorem~\ref{satz4:variationelLoesung} can be proven. First, we show
the result for a smaller class of right-hand sides $\Qc$.
\begin{lemma} \label{lemma4:variationelLoesungL2}
	Let Assumptions~\ref{as:spaces}.\:and \ref{as:OpAB}.\:be
	satisfied
	and let $\gamma$, with $0\le \gamma<\frac{\mu}{C_{V,H}^2}$, be given. Then for $\Qc \in
	L^{2}(0,T;\Hi)$ with $\Qc(t)=\Qc^*(t)$ and $\Qc(t)\geq
	-\gamma^2$ for almost every $t\in [0,T]$ and
	$\Pc_0\in \Hi$ with $\Pc_0=\Pc_0^*$ and $\Pc_0\geq-\gamma
	$, the initial value problem \eqref{eq0:InitialRic} possesses a weak solution
	$\Pc \in \W$ \red{that fulfills $\Pc(t) \geq - \gamma$ for all $t\in [0,T]$}.
\end{lemma}
Note that under the assumptions of
Lemma~\ref{lemma4:variationelLoesungL2} one even obtains that $\Pc'$ is
an element of $L^2(0,T;\V^*)$.
\begin{proof}[Proof of Lemma~\ref{lemma4:variationelLoesungL2}]
	Again, we use the notation from Lemma~\ref{lemma4:konvTF} and its
	proof. In Lemma~\ref{lemma4:konvTF}, we have shown that there exists
	$\Pc \in \W$ such that
	\begin{align*}
	\Pc_{\tau} &\weak \Pc \text{ in } L^2(0,T;\V),\\
	\Pctil_{\tau},\Pc_{\tau} &\weaks \Pc \text{ in } L^{\infty}(0,T;\Hi),
	\quad \text{ and }\\
	\Pctil'_{\tau} &\weak \Pc' \text{ in } L^{2}(0,T;\V^*)
	\end{align*}
	as $\tau \to 0$.
  {This in mind, we start proving that $\Pc$ fulfills the initial
  condition. In order to do so, let $\Rc(t) = \frac{T-t}{T} \, R$ for arbitrary $R\in \V$, $t\in [0,T]$. Then we obtain that
  \begin{align*}
  & - \int_{0}^{T}  \dualb[\V^* \times \V ]{ \Pctil_{\tau}'(t) }{\Rc(t)} dt
  - \int_{0}^{T} \dualb[\V^* \times \V ]{ \Pctil_{\tau}(t) }{\Rc'(t)} dt\\
  &=  - \red{\Big[ }
  \left( \Pctil_{\tau}(T) , \Rc(T) \right) -
  \left( \Pctil_{\tau}(0) , \Rc(0) \right) \red{ \Big] }
  =  \left( P_0 , R \right)
  \end{align*}
  and also
  \begin{align*}
  - \int_{0}^{T}  \dualb[\V^* \times \V ]{ \Pc'(t) }{\Rc(t)} dt
  - \int_{0}^{T} \dualb[\V^* \times \V ]{ \Pc(t) }{\Rc'(t)} dt
  =  \left( \Pc(0) , R\right) .
  \end{align*}
  Since
  \begin{align*}
  &- \int_{0}^{T}  \dualb[\V^* \times \V ]{ \Pctil_{\tau}'(t) }{\Rc(t)} dt
  - \int_{0}^{T} \dualb[\V^* \times \V ]{ \Pctil_{\tau}(t) }{\Rc'(t)} dt\\
  \to &- \int_{0}^{T}  \dualb[\V^* \times \V ]{ \Pc'(t) }{\Rc(t)} dt
  - \int_{0}^{T} \dualb[\V^* \times \V ]{ \Pc(t) }{\Rc'(t)} dt
  \end{align*}
  as $\tau \to 0$, it follows that $\Pc(0) = P_0$.
 }

Using $\Rc(t) = \frac{t}{T} \, R$ for arbitrary $R\in \V$, $t\in [0,T]$, we obtain similarly  that
  \begin{align*}
  &\left( P_N , R \right)
  =\int_{0}^{T}  \dualb[\V^* \times \V ]{ \Pctil_{\tau}'(t) }{\Rc(t)} dt
  + \int_{0}^{T} \dualb[\V^* \times \V ]{ \Pctil_{\tau}(t) }{\Rc'(t)} dt\\
  \to & \int_{0}^{T}  \dualb[\V^* \times \V ]{ \Pc'(t) }{\Rc(t)} dt
  + \int_{0}^{T} \dualb[\V^* \times \V ]{ \Pc(t) }{\Rc'(t)} dt
  = \left( \Pc(T) , R\right) ,
  \end{align*}
  which means that $P_N \weak \Pc(T)$ in $\Hi$ as $\tau \to 0$. This will be employed later in the proof.

	Using Mazur's lemma, see for example \cite[Corollary 3.8]{H.Brezis.2010},
	there exists a sequence
	$\left(\mathcal{K}_{\tau}\right)$ of convex combinations of
	the elements of the sequence $\left(\Pc_{\tau}\right)$ such
	that
	\begin{align*}
	\mathcal{K}_{\tau} &\to \Pc \quad \text{ in } L^2(0,T;\V) \text{ as } \tau \to 0.
	\end{align*}
	Since every function $\Pc_{\tau}$ is piecewise constant
	with values $P_n \geq -\gamma$ ($n = 0, 1, \dots, N$), it is
	clear	that
	$\Pc_{\tau}(t) \geq -\gamma$ for every $t\in [0,T]$. Thus,
	every convex combination of functions $\Pc_{\tau}$ is pointwise greater or
	equal than $-\gamma$.
  Due to the strong convergence of the sequence $(\mathcal{K}_{\tau}
  )$ in $L^2(0,T;\V)$, there exists a subsequence which converges
  pointwise in $\V$ to the limit $\Pc$ almost everywhere in $[0,T]$. Since
  $\mathcal{K}_{\tau}(t) \geq -\gamma$ for every $t \in
  [0,T]$, the limit $\Pc$ possesses the same lower bound almost everywhere.
  \red{Moreover, $\Pc$ is an element of $\W$. Therefore, it is continuous as
  a function with values in $\Hi$ and
  the lower bound is fulfilled for every $t \in [0, T ]$.}

  The sequence $\left(\Ac_{\tau} \Pc_{\tau}\right)$ is bounded in
  $L^2(0,T;\V^*)$, since
  \begin{align*}
    \| \Ac_{\tau}  \Pc_{\tau} \|_{L^2(0,T;\V^*)}
    &= \| \A_{\tau}^*  \Pc_{\tau} + \Pc_{\tau} \A_{\tau} + \Pc_{\tau}^2
    \|_{L^2(0,T;\V^*)}\\
    &\leq \| \A_{\tau}^*  \Pc_{\tau} \|_{L^2(0,T;\V^*)}
    + \|\Pc_{\tau} \A_{\tau} \|_{L^2(0,T;\V^*)}
    + \| \Pc_{\tau}^2  \|_{L^2(0,T;\V^*)}\\
    &\leq 2 \eta \|  \Pc_{\tau} \|_{L^2(0,T;\V)}
    + \sqrt{T} \frac{C_{V,H}}{\sqrt{2}} \| \Pc_{\tau}  \|_{L^{\infty}(0,T;\Hi)}^2,
  \end{align*}
  where $\eta$ is introduced in Assumption~\ref{as:OpAB}. Thus, there
  exists a subsequence that converges weakly to an element $\Sc \in
  L^2(0,T;\V^*)$.
 Since $\Ac_\tau \Pc_\tau = \Qc_\tau - \widehat{P}_\tau'$, it is clear that
 $\Sc$ coincides with $\Qc-\Pc'$.

	Using this subsequence from now on, it remains to prove that
$\Ac \Pc = \A^* \Pc +\Pc \A + \Pc^2$ coincides with the limit $\Sc$ of
$(\Ac_{\tau}
	\Pc_{\tau})$. We show this using the Minty trick (see for
	example \red{\cite[Lemma III.1.3]{GGZ.1974} or \cite[Lemma
	2.13]{T.Roubicek.2013}}). For a function
	$\Rc_{\theta} = \Pc - \theta \Uc$ with $\Uc\in L^{\infty}(0,T;\V)$ and
	$\theta > 0$ small enough, one obtains that
	\begin{align*}
	&\int_{0}^{T} \dual{\Ac_{\tau}(t) \Pc_{\tau}(t) - \Ac_{\tau}(t)
		\Rc_{\theta}(t)}{\Pc_{\tau}(t) - \Rc_{\theta}(t)} dt\\
	&=\int_{0}^{T} \dual{\A^*_{\tau}(t)(\Pc_{\tau}(t) - \Rc_{\theta}(t)) +
		(\Pc_{\tau}(t) - \Rc_{\theta}(t)) \A_{\tau}(t) }{\Pc_{\tau}(t) -
		\Rc_{\theta}(t)} dt\\
	&\qquad +	 \int_{0}^{T} \dual{\Pc_{\tau}^2(t) - \Rc_{\theta}^2(t) }{
		\Pc_{\tau}(t) - \Rc_{\theta}(t)}dt\\
  &\geq \int_{0}^{T} \mu \|\Pc_{\tau}(t) - \Rc_{\theta}(t) \|^2_{\V} dt
  + \int_{0}^{T} \dual{\Pc_{\tau}(t)(\Pc_{\tau}(t) -
    \Rc_{\theta}(t))}{ \Pc_{\tau}(t) - 	\Rc_{\theta}(t)} dt \\
  &\qquad +	 \int_{0}^{T} \dual{(\Pc_{\tau}(t) -
    \Rc_{\theta}(t)) \Pc(t)}{
    \Pc_{\tau}(t) - \Rc_{\theta}(t)}dt\\
  &\qquad - \int_{0}^{T} \theta \dual{(\Pc_{\tau}(t) - \Rc_{\theta}(t))
     \Uc(t)}{ \Pc_{\tau}(t) - \Rc_{\theta}(t)}dt.
  \end{align*}
  An application of Lemma~\ref{lemma3:QuadTermNeg} then yields
  \begin{align*}
    & \dualb{  \Pc_{\tau}(t) (\Pc_{\tau}(t) - \Rc_{\theta}(t))
    }{ \Pc_{\tau}(t) - 	\Rc_{\theta}(t)}\\
    & \geq - \gamma  \| \Pc_{\tau}(t) - \Rc_{\theta}(t) \|_{\Hi}^2
     \geq - \frac{\gamma C_{V,H}^2 }{2} \|\Pc_{\tau}(t) - \Rc_{\theta}(t))
    \|_{\V}^2.
  \end{align*}
  An analogous argument yields that
  \begin{align*}
    &\dual{(\Pc_{\tau}(t) - \Rc_{\theta}(t))\Pc(t)}{  \Pc_{\tau}(t) -
    \Rc_{\theta}(t)}
    \geq - \frac{\gamma C_{V,H}^2 }{2} \|\Pc_{\tau}(t) - \Rc_{\theta}(t))
    \|_{\V}^2.
  \end{align*}
  Altogether, we then get
  \begin{align*}
   &\int_{0}^{T} \dual{\Ac_{\tau}(t) \Pc_{\tau}(t) - \Ac_{\tau}(t)
     \Rc_{\theta}(t)}{\Pc_{\tau}(t) - \Rc_{\theta}(t)} dt\\
  	&\geq \left (\mu- \gamma C_{V,H}^2\right)
	\int_{0}^{T}  \|\Pc_{\tau}(t) - \Rc_{\theta}(t) \|^2_{\V} dt-
	\int_{0}^{T} \theta \|\Pc_{\tau}(t) - \Rc_{\theta}(t)\|^2_{\Hi}
	\|\Uc(t)\|_{\Hi} dt\\
	&\geq \left (\mu- \gamma C_{V,H}^2 - \theta \frac{
	C_{V,H}^2 }{2} \|\Uc\|_{L^{\infty}(0,T;\Hi)} \right)
	\int_{0}^{T} \|\Pc_{\tau}(t) - \Rc_{\theta}(t)
	\|^2_{\V} dt
	\geq 0
	\end{align*}
	for $\theta \frac{C_{V,H}^2 }{2} \|\Uc\|_{L^{\infty}(0,T;\Hi)} \leq
	\mu- 	\gamma C_{V,H}^2$.
 Thus,
	\begin{align*}
	&\int_{0}^{T} \dual{\Ac_{\tau}(t)\Pc_{\tau}(t)}{\Pc_{\tau}(t)} dt\\
	&\geq \int_{0}^{T} \dual{\Ac_{\tau}(t) \Rc_{\theta}(t) }{\Pc_{\tau}(t) -
		\Rc_{\theta}(t) } dt + \int_{0}^{T} \dual{\Ac_{\tau}(t)\Pc_{\tau}(t)}{
		\Rc_{\theta}(t)} dt
	\end{align*}
	holds for $\theta >0$ sufficiently small. Let us look at the summands of the
	right-hand side separately. For the second summand, the weak
	convergence of
	$( \Ac_{\tau} \Pc_{\tau} )$ implies that
	\begin{align*}
	\int_{0}^{T} \dual{\Ac_{\tau}(t)\Pc_{\tau}(t)}{\Rc_{\theta}(t)} dt \to
	\int_{0}^{T} \dual{\Sc(t)}{\Rc_{\theta}(t)} dt
	\end{align*}
	as $\tau \to 0$. To study the convergence of the first summand, we split it
	up as
	\begin{align} \label{eq4:Minty1}
	&\int_{0}^{T} \dual{\Ac_{\tau}(t)\Rc_{\theta}(t)}{\Pc_{\tau}(t) -
		\Rc_{\theta}(t)} dt \nonumber \\
	\begin{split}
&=
\int_{0}^{T} \dual{\A^*_{\tau}(t)\Rc_{\theta}(t)}{\Pc_{\tau}(t)} dt
+ \int_{0}^{T} \dual{\Rc_{\theta}(t) \A_{\tau}(t)}{\Pc_{\tau}(t)} dt
+ \int_{0}^{T} \dual{\Rc_{\theta}^2(t)}{\Pc_{\tau}(t)} dt
\\
&\quad
- \int_{0}^{T} \dual{\A^*_{\tau}(t)\Rc_{\theta}(t) + \Rc_{\theta}(t)
		\A_{\tau}(t) + \Rc_{\theta}^2(t) }{\Rc_{\theta}(t)}
dt
	\end{split}
	\end{align}
	Then for the first summand on the right-hand side of \eqref{eq4:Minty1},
	we use
  {\begin{align}
    \notag
  & \Big|\int_{0}^{T} \dual{\A^*_{\tau}(t)\Rc_{\theta}(t)}{\Pc_{\tau}(t)} dt -
  \int_{0}^{T} 	\dual{\A^*(t)\Rc_{\theta}(t)}{\Pc(t)} dt \Big| \\
    \notag
  & \leq \Big| \int_{0}^{T} \dual{(\A^*_{\tau}(t) - \A^*(t))\Rc_{\theta}(t)
  }{\Pc_{\tau}(t)} dt \Big|
 + \Big| \int_{0}^{T}  \dual{\A^*(t)\Rc_{\theta}(t)}{\Pc_{\tau}(t) - \Pc(t)}
   dt \Big| \\
   \begin{split} \label{eq4:Minty4}
   & \leq \left( \int_{0}^{T} \|(\A^*_{\tau}(t) -
   \A^*(t))\Rc_{\theta}(t)\|_{\V^*}^2dt  \int_{0}^{T} \|
   \Pc_{\tau}(t)\|_{\V}^2dt\right)^{\frac{1}{2}}\\
   &\quad+ \Big| \int_{0}^{T} \dual{\A^*(t)\Rc_{\theta}(t)}{\Pc_{\tau}(t) - \Pc(t)}
   dt \Big|.
   \end{split}
  \end{align}
  The second summand on the right-hand side of \eqref{eq4:Minty4} converges to zero due
  to the weak convergence of
  $(\Pc_{\tau} )$ towards $\Pc$ in $L^2(0,T;\V)$. The weak convergence of the sequence
  $(\Pc_{\tau} )$ in $L^2(0,T;\V)$ also implies that it is bounded. As shown in
  Lemma~\ref{lemma4:KonvergenzA},  $(\A^*_{\tau}(t) -
  \A^*(t))\Rc_{\theta}(t) \to 0$ in $\V^*$ for almost every $t \in [0,T]$ as $\tau
  \to 0$. Due to the uniform boundedness of $\A^*$, we obtain the
  following bound
  \begin{align*}
  \|(\A^*_{\tau}(t) -  \A^*(t)) \Rc_{\theta}(t)\|_{\V^*}^2
  \leq 4\eta \|\Rc_{\theta}(t)\|_{\V}^2 \quad \text{ for almost every } t\in [0,T],
  \end{align*}
  with $4\eta \|\Rc_{\theta}\|_{\V}^2 \in L^1(0,T)$. Therefore, the first
  summand on the right-hand side of \eqref{eq4:Minty4} converges to zero using Lebesgue's
  theorem on dominated convergence. This proves
  \begin{align*}
  \int_{0}^{T} \dual{\A^*_{\tau}(t)\Rc_{\theta}(t)}{\Pc_{\tau}(t)} dt
  \to
  \int_{0}^{T} 	\dual{\A^*(t)\Rc_{\theta}(t)}{\Pc(t)} dt
  \end{align*}
  as $\tau \to 0$.
  For the second summand on the right-hand side of \eqref{eq4:Minty1},
  we can follow an analogous argumentation. In this case, we have
  \begin{align*}
  & \Big|\int_{0}^{T} \dual{\Rc_{\theta}(t)\A_{\tau}(t)}{\Pc_{\tau}(t)} dt -
  \int_{0}^{T} 	\dual{\Rc_{\theta}(t)\A(t)}{\Pc(t)} dt \Big| \\
  & \leq \Big| \int_{0}^{T} \dual{\Rc_{\theta}(t) (\A_{\tau}(t) - \A(t))
  }{\Pc_{\tau}(t)} dt \Big|
  + \Big| \int_{0}^{T}  \dual{\Rc_{\theta}(t)\A(t)}{\Pc_{\tau}(t) - \Pc(t)}
  dt \Big| \\
  & \leq \left( \int_{0}^{T} \|\Rc_{\theta}(t) (\A_{\tau}(t) -
  \A(t))\|_{\V^*}^2dt  \int_{0}^{T} \|
  \Pc_{\tau}(t)\|_{\V}^2dt\right)^{\frac{1}{2}}\\
  &\quad+ \Big| \int_{0}^{T} \dual{\Rc_{\theta}(t)\A(t)}{\Pc_{\tau}(t) - \Pc(t)}
  dt \Big|.
  \end{align*}
  As before, the weak convergence of $(\Pc_{\tau} )$ towards $\Pc$ in $L^2(0,T;\V)$ and
  Lemma~\ref{lemma4:KonvergenzA} imply that
\[
  \int_{0}^{T}
  \dual{\Rc_{\theta}(t)\A_{\tau}(t)}{\Pc_{\tau}(t)} dt \to
  \int_{0}^{T} \dual{\Rc_{\theta}(t)\A(t)}{\Pc(t)} dt
\]
  as $\tau \to 0$.

  The third summand on the right-hand side of \eqref{eq4:Minty1} converges to $\int_{0}^{T}
  \dual{\Rc_{\theta}^2(t)}{\Pc(t)} dt$ as $\tau \to 0$, since $(\Pc_{\tau} )$
  converges weakly to $\Pc$ in $L^2(0,T;\V)$.
  }

For the fourth term on the right-hand side of \eqref{eq4:Minty1}, one similarly shows convergence to $- \int_0^T \langle \Ac(t)\Rc_\theta(t),
\Rc_\theta(t)\rangle dt$ as $\tau \to 0$.

	Therefore, we obtain that
	\begin{align*}
	\lim\limits_{\tau \to 0}\int_{0}^{T}
	\dual{\Ac_{\tau}(t)\Rc_{\theta}(t)}{\Pc_{\tau}(t) - \Rc_{\theta}(t)} dt
	= \int_{0}^{T} \dual{\Ac(t)\Rc_{\theta}(t)}{\Pc(t) - \Rc_{\theta}(t)} dt.
	\end{align*}
	Altogether, this yields the estimate
	\begin{align}\label{eq4:Minty2}
	&\liminf_{\tau \to 0} \int_{0}^{T}
	\dual{\Ac_{\tau}(t)\Pc_{\tau}(t)}{\Pc_{\tau}(t)} dt\\
	&\geq \int_{0}^{T} \dual{\Ac(t)\Rc_{\theta}(t)}{\Pc(t) - \Rc_{\theta}(t)} dt
	+ \int_{0}^{T} \dual{\Sc(t)}{\Rc_{\theta}(t)} dt. \nonumber
	\end{align}
	Using \eqref{eq4:RiccDiscret}, we obtain that
	\begin{align*}
	\int_{0}^{T} \dual{\Ac_{\tau}(t)\Pc_{\tau}(t)}{\Pc_{\tau}(t)} dt =
	\int_{0}^{T} \big( \dual{\Qc_{\tau}(t)}{\Pc_{\tau}(t)} -
	\big\langle \Pctil'_{\tau}(t) , \Pc_{\tau}(t) \big\rangle \big) dt.
	\end{align*}
	As $\Qc_{\tau}$ converges strongly  to $\Qc$ in $L^2(0,T;\Hi)$ and thus in $L^2(0,T;\V^*)$, we have that
	\begin{align*}
	\int_{0}^{T} \dual{\Qc_{\tau}(t)}{\Pc_{\tau}(t)} dt \to \int_{0}^{T}
	\dual{\Qc(t)}{\Pc(t)} dt
	\end{align*}
	holds as $\tau \to 0$. Furthermore,
	\begin{align*}
	\int_{0}^{T}  \dual{\Pctil'_{\tau}(t)}{\Pc_{\tau}(t)} dt
  &= \sum_{n=1}^{N} \frac{1}{\tau} \int_{t_{n-1}}^{t_n} dt \dual{P_n -
  P_{n-1}}{P_n}\\
	&\geq \frac{1}{2} \sum_{n=1}^{N} \big( \|P_n\|_{\Hi}^2 - \|P_{n-1}
	\|_{\Hi}^2\big)
  = \frac{1}{2} \|P_N\|_{\Hi}^2 - \frac{1}{2} \|P_0\|_{\Hi}^2,
	\end{align*}
	{together with the fact that $P_0 = \Pc(0)$ as well as $P_N \weak
	\Pc(T)$ in $\Hi$ as $\tau \to 0$ and the weak lower
	semi-continuity of the norm}, yields that
	\begin{align*}
	\liminf_{\tau \to 0} \int_{0}^{T}  \dual{\Pctil'_{\tau}(t)}{\Pc_{\tau}(t)} dt
	&\geq \frac{1}{2} \|\Pc(T)\|_{\Hi}^2 - \frac{1}{2} \|\Pc(0)\|_{\Hi}^2
  = \int_{0}^{T} \dual{\Pc'(t)}{\Pc(t)} dt.
  \end{align*}
  For the last step, we use that $\Pc \in \W$,
  compare, e.g., \cite[Section 20]{L.Tartar.2006}.
	Altogether, this yields that
	\begin{align*}
	\limsup_{\tau \to 0} \int_{0}^{T}
	\dual{\Ac_{\tau}(t)\Pc_{\tau}(t)}{\Pc_{\tau}(t)} dt
	\leq \int_{0}^{T} \dual{\Qc(t)}{\Pc(t)} dt -
	\int_{0}^{T}\dual{\Pc'(t)}{\Pc(t)}dt.
	\end{align*}
	Using this estimate and \eqref{eq4:Minty2}, we obtain that
	\begin{align*}
	\int_{0}^{T}\dual{\Sc(t)}{\Pc(t)} dt &= \int_{0}^{T} \dual{\Qc(t) -
		\Pc'(t)}{\Pc(t)} dt\\
	& \geq \int_{0}^{T} \big( \dual{\Ac(t) \Rc_{\theta}(t)}{\Pc(t) - \Rc_{\theta}(t)}
	+ \dual{\Sc(t)}{\Rc_{\theta}(t)} \big)dt
	\end{align*}
	which can be rewritten as
	\begin{equation} \label{eq4:Minty3}
	\int_{0}^{T} \dual{\Sc(t)}{\Pc(t) - \Rc_{\theta}(t)} dt \geq \int_{0}^{T}
	\dual{\Ac(t) \Rc_{\theta}(t)}{\Pc(t) - \Rc_{\theta}(t)}dt.
	\end{equation}
	Reinserting $\Rc_{\theta} = \Pc - \theta \Uc$ for $\theta >0$ small enough,
	the estimate has the form
	\begin{align*}
	\theta \int_{0}^{T} \dual{\Sc(t)}{\Uc(t)} dt \geq \theta \int_{0}^{T}
	\dual{\Ac(t) (\Pc(t) - \theta \Uc(t))}{\Uc(t)}dt.
	\end{align*}
	Dividing this by $\theta$ and followed by the limiting process $\theta
	\to 0$, this yields
	\begin{align*}
	\int_{0}^{T} \dual{\Sc(t)}{\Uc(t)} dt \geq
	\int_{0}^{T} \dual{ \Ac(t)\Pc(t)}{\Uc(t)} dt
	\end{align*}
for all $\Uc$ and thus $\Sc = \Ac \Pc$.
\end{proof}

The last lemma in mind, the convergence of the time discretization for a
right-hand side $\Qc \in L^1(0,T;\Hi) + L^2(0,T;\V^*)$ can now be deduced.

\begin{proof}[Proof of Theorem~\ref{satz4:variationelLoesung}]
	For a right-hand side $\Qc = \Qc_1 + \Qc_2 \in L^1(0,T;\Hi) + L^2(0,T;\V^*)$,
	there exist sequences $\left(\Qc_{1,i}\right)_{i\in \N}$
	and $\left(\Qc_{2,i} \right)_{i\in \N}$ in $L^2(0,T;\Hi)$ such that
	\begin{align*}
	\Qc_{1,i} \to \Qc_1 \text{ in } L^1(0,T;\Hi) \quad \text{ and } \quad
	\Qc_{2,i} \to \Qc_2 \text{ in } L^2(0,T;\V^*)
	\end{align*}
	as $i \to \infty$. Furthermore, we set $\Qc_i = \Qc_{1,i} +
	\Qc_{2,i}$ for $i \in \N$.
	Every problem
	\begin{equation}\label{eq4:ricEqApprox}
	\begin{split}
	\Pc_i'(t) + \A^*(t)\Pc_i(t)+ \Pc_i(t)\A(t) + \Pc_i^2(t) &= \Qc_i(t), \quad
	t\in (0,T),\\
	\Pc_i(0) &= P_0
	\end{split}
	\end{equation}
	has a solution  $\Pc_i \in \W$ \red{with $\Pc_i(t) \geq - \gamma$ for every
	$t \in [0,T]$}.

	For arbitrary $i,j \in \N$, we consider the difference of the solutions
	$\Pc_i$
	and $\Pc_j$ of the associated problems \eqref{eq4:ricEqApprox}. Using Lemma~\ref{lemma3:QuadTermNeg},  it follows that
	\begin{align}
	& \frac{1}{2} \frac{d}{dt} \|\Pc_i(t) - \Pc_j(t) \|_{\Hi}^2  + \mu\|\Pc_i(t)
	- 	\Pc_j(t)\|^2_{\V}
	- \gamma C_{V,H}^2 \|\Pc_i(t) - \Pc_j(t) \|_{\V}^2 \nonumber\\
	&\leq \ska{\Pc_i'(t) - \Pc_j'(t)}{\Pc_i(t) - \Pc_j(t)} \nonumber\\
	&\qquad + \dual{\A^*(t)(\Pc_i(t) - \Pc_j(t))+ (\Pc_i(t) - \Pc_j(t))\A(t)
	}{\Pc_i(t) - \Pc_j(t)} \nonumber\\
	&\qquad + \ska{\Pc_i^2(t) - \Pc_j^2(t)}{\Pc_i(t) - \Pc_j(t)} \nonumber\\
	&= \ska{\Qc_{1,i} (t) - \Qc_{1,j	}(t)}{\Pc_i(t) - \Pc_j(t)} +
	\ska{\Qc_{2,i}
		(t) - \Qc_{2,j}(t)}{\Pc_i(t) - \Pc_j(t)}
\label{hh2}
	\end{align}
for almost every $t\in (0,T)$.
	Integrating this estimate and  applying  Young's inequality yields
	\begin{equation} \label{eq4:integratedApriori}
	\begin{split}
	& \frac{1}{2}\|\Pc_i(t) - \Pc_j(t) \|_{\Hi}^2 + \left(\mu-
	\gamma C_{V,H}^2 \right) \int_{0}^{t}
	\|\Pc_i(s) - \Pc_j(s) \|_{\V}^2 ds\\
	&\leq  \max_{t\in [0,T]} \|\Pc_i(t) - \Pc_j(t)\|_{\Hi} \int_{0}^{t}
	\|\Qc_{1,i} (s) - \Qc_{1,j}(s)\|_{\Hi}ds\\
	&\quad + \int_{0}^{t} \Big( \frac{1}{2\big(\mu-
		\gamma C_{V,H}^2\big)} \|\Qc_{2,i}(s) -
	\Qc_{2,j}(s)\|_{\V^*}^2 + \frac{\mu-
		\gamma C_{V,H}^2}{2} \|\Pc_i(s) - \Pc_j(s)\|_{\V}^2\Big) ds.
	\end{split}
	\end{equation}
	As $\Pc_i, \Pc_j \in C([0,T];\Hi)$, there exists $t_0\in[0,T]$ such that
	\begin{align*}
	\|\Pc_i(t_0) - \Pc_j(t_0) \|_{\Hi} = \max_{t\in[0,T]} \|\Pc_i(t) - \Pc_j(t)
	\|_{\Hi}
	\end{align*}
	holds. As $t \in [0,T]$ in estimate \eqref{eq4:integratedApriori} is
	arbitrary, we can take $t = t_0$ and obtain that
	\begin{align} \label{eq4:contbound}
  \begin{split}
    \|\Pc_i(t_0) - \Pc_j(t_0) \|_{\Hi}^2
    &\leq  2 \|\Pc_i(t_0) - \Pc_j(t_0)\|_{\Hi} \int_{0}^{T} \|\Qc_{1,i} (s) -
    \Qc_{1,j}(s)\|_{\Hi}ds\\
    &\qquad  + \frac{1}{\mu-
      \gamma C_{V,H}^2 } \int_{0}^{T}\|\Qc_{2,i}(s) -
    \Qc_{2,j}(s)\|_{\V^*}^2 ds.
  \end{split}
	\end{align}
	Setting
	\begin{align*}
	x &= \|\Pc_i(t_0) - \Pc_j(t_0) \|_{\Hi},\\
	a&= \int_{0}^{T} \|\Qc_{1,i} (s) - \Qc_{1,j}(s)\|_{\Hi}ds, \quad \text{ and	}\\
	b &= \Big( \frac{1}{\mu- \gamma C_{V,H}^2} \int_{0}^{T}\|\Qc_{2,i}(s)
	- 	\Qc_{2,j}(s)\|_{\V^*}^2 ds \Big)^{\frac{1}{2}}
	\end{align*}
	for abbreviation, \eqref{eq4:contbound} is equivalent to
	\begin{align*}
	x^2 \leq 2ax + b^2.
	\end{align*}
	This estimate implies that
	\begin{align*}
	(x-a)^2 = x^2 -2ax +a^2 \leq a^2 +b^2.
	\end{align*}
Taking the square root on both sides, this yields
	\begin{align*}
	x-a \leq \sqrt{ a^2 +b^2} \leq a + b.
	\end{align*}
	Altogether this leads to the estimate
	\begin{align*}
	x \leq 	2a + b,
	\end{align*}
	which implies that
	\begin{align*}
	&\max_{t\in[0,T]} \|\Pc_i(t) - \Pc_j(t) \|_{\Hi}\\
	&\leq 2\int_{0}^{T} \|\Qc_{1,i} (s) - \Qc_{1,j}(s)\|_{\Hi}ds +
	\left(\frac{1}{\mu- \gamma C_{V,H} }\int_{0}^{T}\|\Qc_{2,i}(s)
		- \Qc_{2,j}(s)\|_{\V^*}^2 ds\right)^{\frac{1}{2}} =: C_{max}.
	\end{align*}
Using this estimate in
	\eqref{eq4:integratedApriori}, we further obtain that for all $t\in [0,T]$
	\begin{align*}
	& \|\Pc_i(t) - \Pc_j(t) \|_{\Hi}^2 + \left( \mu-
	\gamma C_{V,H}^2  \right) \int_{0}^{t} \|\Pc_i(s)
	- \Pc_j(s) \|_{\V}^2 ds\\
	&\leq  2  C_{max} \int_{0}^{T}\|\Qc_{1,i} (s) - \Qc_{1,j}(s)\|_{\Hi}ds\\
	&\quad  + \frac{1}{\mu-
		\gamma C_{V,H} } \int_{0}^{T}  \|\Qc_{2,i}(s) -
	\Qc_{2,j}(s)\|_{\V^*}^2 ds.
	\end{align*}
	This proves that $\seqd{\Pc}{i}$ is a Cauchy sequence in $C([0,T];\Hi) \cap
	L^2(0,T;\V)$ and therefore convergent to a certain limit $\Pc \in C([0,T];\Hi) \cap
	L^2(0,T;\V)$. This convergence and the estimate
	\begin{align*}
	&\quad \ \int_{0}^{T} \| \Ac(t) \Pc_i(t) - \Ac(t) \Pc(t)\|^2_{\V^*} dt\\
	&\leq c \int_{0}^{T} \big( \eta \left\| \Pc_i(t) - \Pc(t) \right\|^2_{\V}
	+  C_{V,H}^4 \|\Pc_i(t) -  \Pc(t) \|^2_{\V}
	\left(\| \Pc_i(t) \|^2_{\Hi} + \| \Pc(t) \|^2_{\Hi}\right) \big) dt,
	\end{align*}
  where $\eta$ is defined in Assumption~\ref{as:OpAB}.,
	imply the convergence of the sequence $ (\Ac \Pc_i)_{i\in \N}$ to $ \Ac \Pc$
	in $L^2(0,T;\V^*)$. As the sequence $\seqd{\Qc}{i}$ of right-hand sides
	converges strongly in  $L^1(0,T;\Hi) + L^2(0,T;\V^*)$ towards $\Qc$, it follows that
	\begin{align*}
	\Pc_i' = \Qc_i - \Ac \Pc_i
 \to \Qc - \Ac \Pc \text{ in } L^1(0,T;\Hi) + L^2(0,T;\V^*) \, .
	\end{align*}
This shows that $\Pc$ possesses a weak derivative that coincides with $\Qc
- \Ac \Pc$. Moreover, $\Pc \in \W \hookrightarrow C([0,T];\Hi)$ \red{with
$\Pc(t) \geq -\gamma$ for every $t \in [0,T]$}. Considering the continuity of
the corresponding trace operator mapping onto the evaluation at $t=0$
immediately shows that $\Pc(0)$ is the limit of $\Pc_i(0) = P_0$ in $\Hi$.

\red{
Finally, let us prove the uniqueness of the solution within the class of
functions in $\W$ which fulfill $\Pc(t) \geq -\gamma$ for every $t\in [0,T]$.
In order to do so let $\Pc_1$ and $\Pc_2$ be such solutions. Then we
consider the difference of the associated problems \eqref{eq0:InitialRic} to
the solutions $\Pc_1$ and $\Pc_2$ for $t \in [0,T]$
\begin{align*}
  \Pc_1'(t) - \Pc_2'(t) + \A^*(t)(\Pc_1(t) - \Pc_2(t)) + (\Pc_1(t) - \Pc_2(t))\A(t)
  + \Pc_1^2(t) - \Pc_2^2(t) = 0.
\end{align*}
Testing this equation with $\Pc_1(t) - \Pc_2(t)$ and using both the strong
monotonicity as well as Lemma~\ref{lemma3:QuadTermNeg}, we find that
for $t \in [0,T]$,
\begin{align*}
  \frac{1}{2} \frac{d}{dt} \|\Pc_1(t) - \Pc_2(t)\|_{\Hi}^2
  + \mu \| \Pc_1(t) - \Pc_2(t) \|_{\V}^2
  - 2 \gamma \| \Pc_1(t) - \Pc_2(t) \|_{\Hi}^2 \leq 0.
\end{align*}
Solving this differential inequality shows that
\begin{align*}
  \| \Pc_1(t) - \Pc_2(t) \|_{\Hi}^2
  \leq \mathrm{e}^{4\gamma t} \| \Pc_1(0) - \Pc_2(0) \|_{\Hi}^2
\end{align*}
and thus implies the uniqueness.}
\end{proof}

So far we have proven that the sequence of interpolations
\eqref{eq4:defInterpol} has a weakly
convergent subsequence. This convergence result can be strengthened.

\begin{satz}\label{satz4:strongConv}
	Let the assumptions of Theorem~\ref{satz4:variationelLoesung} be satisfied.
	Further, let $(N_k)_{k\in\N}$ be a sequence of positive integers such that
	$N_k \to \infty$ as $k \to \infty$ and let $(\tau_k)_{k\in\N}$ be the
	sequence of step sizes $\tau_k = \frac{T}{N_k}$ such that
	$\sup_{k\in\N}\tau_k < \frac{\mu}{2C_{V,H}^2}$. \red{Then the sequence
	$\left(	\Pc_{\tau_{k}} \right)_{k\in \N}$ of piecewise constant approximate
	solutions converges	strongly in $L^2(0,T;\V)$ towards the solution $\Pc
	\in \W$ that fulfills $\Pc(t) \geq -\gamma$ for every $t \in [0,T]$ of the
	initial value 	problem \eqref{eq0:InitialRic} as $k \to \infty$.}
\end{satz}

\begin{proof}
\red{Using the weakly convergent subsequence from the proof of
Theorem~\ref{satz4:variationelLoesung}, we can write (omitting the index
$k$),}
	\begin{align*}
	& \left(\mu - \gamma C_{V,H}^2 \right) \int_{0}^{T} \|
	\Pc_{\tau}(t) - \Pc(t)\|_{\V}^2 dt \notag \\
	&\leq \int_{0}^{T} \dual{\A^*_{\tau}(t) (\Pc_{\tau}(t) -  \Pc(t)) +
		(\Pc_{\tau}(t) -  \Pc(t))\A_{\tau}(t) }{\Pc_{\tau}(t) - \Pc(t)} dt \notag \\
	&\quad + \int_{0}^{T} \dual{ \Pc_{\tau}^2(t)- \Pc^2(t)}{\Pc_{\tau}(t) - \Pc(t)}
	dt \notag\\
	&= \int_{0}^{T} \dual{\Ac_{\tau}(t) \Pc_{\tau}(t)}{\Pc_{\tau}(t) - \Pc(t)} -
	\dual{\Ac_{\tau}(t) \Pc(t)}{\Pc_{\tau}(t) - \Pc(t)} dt. \label{eq4:strongConv1}
	\end{align*}
In the proof of Lemma~\ref{lemma4:variationelLoesungL2}, it is shown with Minty's trick that
$$
\int_{0}^{T} \dual{\Ac_{\tau}(t) \Pc_{\tau}(t)}{\Pc_{\tau}(t)}
 dt \to \int_0^T \dual{\Ac(t) \Pc(t)}{\Pc(t)} dt .
$$
This together with the weak convergence of $\Pc_\tau$ towards $\Pc$ in
$L^2(0,T;\V)$ and the convergence of $\Ac_{\tau} \Pc$ towards $\Ac \Pc$
in  $L^2(0,T;\V^*)$ implies the assertion.
\red{As the solution is unique within the class of functions in $\W$ that
fulfill $\Pc(t) \geq -\gamma$ for every $t \in [0,T]$, one can argue that the
entire sequence converges using the subsequence principle.}
\end{proof}

\section*{Acknowledgments}
The authors would like to thank Christian Kreusler (Berlin) for helpful discussions and suggestions as well as careful reading of the manuscript.

\end{document}